\documentclass[11pt, reqno]{amsart}
\usepackage[dvipsnames,svgnames,x11names]{xcolor}
\usepackage[mono=false]{libertine}
\usepackage[margin=2.5cm]{geometry}
\usepackage[markup=underlined]{changes}
\definechangesauthor[color=NavyBlue]{ale}

\usepackage{amsfonts}
\usepackage{amsthm,amsmath}
\usepackage{amsbsy}
\usepackage{bm}
\usepackage{amssymb} 
\usepackage{verbatim}
\usepackage{wrapfig}

\usepackage{times}
\usepackage{pgf,tikz}
\usepackage{graphicx}
\usepackage{tgpagella}
\usepackage{mathpazo}
\usetikzlibrary{trees,shadows}
\usetikzlibrary{arrows}
\usepackage{units}
\usepackage{enumerate,enumitem}
\usepackage{bbm}
\usepackage{exscale}
\usepackage[hyperindex,breaklinks]{hyperref}
\hypersetup{ colorlinks=true, linkcolor=blue, citecolor=blue,
  filecolor=blue, urlcolor=blue}
  \numberwithin{equation}{section} 

 \makeatletter
\def\paragraph{\@startsection{paragraph}{4}%
  \z@\z@{-\fontdimen2\font}%
  {\normalfont\itshape}}
  \makeatother

\DeclareMathOperator{\one}{\mathbbm{1}} 

\usepackage[comma, sort&compress]{natbib}
\renewcommand{\c}{\mathrm{c}}
\theoremstyle{plain} 
    \newtheorem{theorem}{Theorem}
    \newtheorem{lemma}[theorem]{Lemma}

    \newtheorem{claim}[theorem]{Claim}

\theoremstyle{definition} 
    \newtheorem{definition}[theorem]{Definition}
    \newtheorem{fact}[theorem]{Fact}
    
    \newtheorem{remark}[theorem]{Remark}

\DeclareMathOperator{\T}{\mathbb{T}}
\DeclareMathOperator{\R}{\mathbb{R}}

\DeclareMathOperator{\Z}{\mathbb{Z}}
\DeclareMathOperator{\N}{\mathbb{N}}

\usepackage{mathtools}
\DeclarePairedDelimiter\ceil{\lceil}{\rceil}
\DeclarePairedDelimiter\floor{\lfloor}{\rfloor}

\newcommand{\e}{\mathrm{e}}

\newcommand{\eps}{\varepsilon}
\newcommand{\vr}{\varphi}

\newcommand{\z}{\mathbf{z}}

\newcommand{\prob}{\mathbf{P}}
\newcommand{\E}{\mathbf{E}}

\newcommand{\la}{\left\langle}
\newcommand{\ra}{\right\rangle}

\definecolor{xdxdff}{rgb}{0.49019607843137253,0.49019607843137253,1}
\definecolor{ffffff}{rgb}{1,1,1}
\definecolor{ududff}{rgb}{0.30196078431372547,0.30196078431372547,1}
\definecolor{zzttqq}{rgb}{0.6,0.2,0}

\begin{document}
\title{Maximum of the membrane model on regular trees}

\author[A. Cipriani]{Alessandra Cipriani}
\address{TU Delft (DIAM), Building 28, van Mourik Broekmanweg 6, 2628 XE, Delft, The Netherlands, \& Department of Statistical Science, UCL, 1-19 Torrington Place, London, WC1E 7HB, UK }
\email{a.cipriani@ucl.ac.uk}
\author[B. Dan]{Biltu Dan}
\address{Department of Mathematics, Indian Institute of Science, Bangalore - 560012, India }
\email{biltudanmath@gmail.com}
\author[R.~S.~Hazra]{Rajat Subhra Hazra}
\address{University of Leiden, Niels Bohrweg 1, 2333 CA, Leiden, The Netherlands \\
\& Theoretical Statistics and Mathematics Unit, Indian Statistical Institute, Kolkata}
\email{r.s.hazra@math.leidenuniv.nl}
\author[R. Ray]{Rounak Ray}
\address{Eindhoven University of Technology, P.O. Box 513, 5600 MB, Eindhoven, The Netherlands}
\email{r.ray@tue.nl}

\date{\today}

\begin{abstract}
The discrete membrane model is a Gaussian random interface whose inverse covariance is given by the discrete biharmonic operator on a graph. In literature almost all works have considered the field as indexed over $\Z^d$, and this enabled one to study the model using methods from partial differential equations. In this article we would like to investigate the dependence of the membrane model on a different geometry, namely trees. The covariance is expressed via a random walk representation which was first determined by \cite{vanderbei}. We exploit this representation on $m$-regular trees and show that the infinite volume limit on the infinite tree exists when $m\ge 3$. Further we determine the behavior of the maximum under the infinite and finite volume measures. 
\end{abstract}
\maketitle
\section{Introduction}
The main object of study in this article is the membrane model (MM), also known as discrete bilaplacian or biharmonic model. As a random interface, the MM can be defined as a collection of Gaussian heights indexed over a graph. In this article, we will study the MM on regular trees. Let $\T_m$ be an $m$-regular infinite tree, that is, a rooted tree with the root having $m$-children and each of the children thereafter having $m-1$ children. With abuse of notation we will denote the vertex set of $\T_m$ by $\T_m$ itself. Then the MM is defined to be a Gaussian field $\varphi=(\varphi_{x})_{x\in\T_m}$, whose distribution is determined by the probability measure on $\mathbb{R}^{\T_m}$ with density
\begin{align}\label{eq:prob_measure}
\prob_{\Lambda}(\mathrm{d}\varphi):=\frac1{Z_\Lambda} \exp\left(-\frac{1}{2}\sum_{x\in\T_m}(\Delta\varphi_{x})^2 \right)\prod_{x\in{\Lambda}}\mathrm{d}\varphi_{x}\prod_{x\in\T_m\setminus \Lambda}\delta_{0}(\mathrm{d}\varphi_{x}).
\end{align}
Here $\Lambda\subset\T_m$ is a finite subset, $\Delta$ is the discrete Laplacian defined by
\begin{align} \label{eq:def_Delta}
   \Delta f_x:=\Delta f(x):= \sum\limits_{y\sim x} \frac{1}{m}(f(y)-f(x)),\quad f:\T_m\to \R, \;x\in\T_m,
\end{align}
where $y\sim x$ means that $y$ is a neighbor of $x$, $\mathrm{d}\varphi_{x}$ is the Lebesgue measure on $\R$, $\delta_{0}$ is the Dirac measure at $0,$ and $Z_{\Lambda}$ is a normalising constant. We are imposing zero boundary conditions i.e. almost surely $\varphi_{x}=0$ for all $x\in\T_m\setminus{\Lambda}$, but the definition
holds for more general boundary conditions.

 The membrane model was introduced and studied mostly in the case $\Lambda\subset \Z^d$. For example, the existence of an infinite volume measure for $d\ge 5$ was proved in \cite{Sakagawa} and later the model and its properties were studied in details in \cite{Kurt_thesis}. The point process convergence of extremes on $\Z^d$ for $d\ge 5$ was dealt with in \cite{CCHInterfaces}. The case of $d=4$ is related to log-correlated models and the limit of the extremes was derived in \cite{schweiger:2019}. Finally the scaling limit of the maximum in lower dimensions follows from the scaling limit of the model which was obtained by \cite{CarJDScaling} in $d=1$ and by~\cite{mm_scaling} in $d=2,\,3$. 

The discrete Gaussian free field (DGFF) is a well studied example of a discrete interface model and has connections to other stochastic processes, such as branching random walk and cover times. Most of these connections arise due to the fact that the covariance of the DGFF is the Green's function of the simple random walk. This is not the case for the MM, essentially because the biharmonic operator does not satisfy a maximum principle. This also depends heavily on the boundary conditions: closed formulas for the bilaplacian covariance matrix have been found~\citep{Kurt_thesis,hirschler2021laplace,hirschler2020polyharmonic}, however they do not apply to our choice of boundary values. On the square lattice one can rely on other techniques, namely discrete PDEs, to prove results in the bilaplacian case. However as soon as one goes beyond $\Z^d$ approximations of boundary value problems are less straightforward, and our work is prompted from this aspect. We will use a probabilistic solution of the Dirichlet problem for the bilaplacian \citep{vanderbei} to investigate the membrane model indexed on regular trees. We restrict our study to regular trees because these graphs have many features which are different from $\Z^d$. One of the most striking contrasts is that the number of vertices in the $n$-th generation is comparable to the size of the graph up to the $n$-th generation. From Vanderbei's representation, it is clear that the boundary plays a prominent role in the behavior of the covariance structure. {We will use this representation to derive the maximum of the field under the infinite and finite volume measures. In the next section we describe our set-up and also state the main results, followed by a discussion on future directions.} 

\paragraph{\bf Acknowledgement.}
Part of this work was carried out when BD, RSH and RR were at Indian Statistical Institute, Kolkata. They thank the institute for the great hospitality. RSH thanks Sayan Das for sharing an estimate for Lemma~\ref{lem:C_k}.  The authors would like to thank two anonymous referees for several comments and suggestions that improved considerably the article. In particular we would like to thank one referee for suggesting us the proof presented in Appendix~\ref{app:alternative}.

AC is supported by the grant 613.009.102 of the Netherlands Organisation for Scientific Research (NWO). BD is supported by IISc through C. V. Raman postdoctoral fellowship and RSH was supported by the DST-MATRICS grant.\begin{wrapfigure}{r}{0.08\textwidth}
    \includegraphics[width=0.08\textwidth]{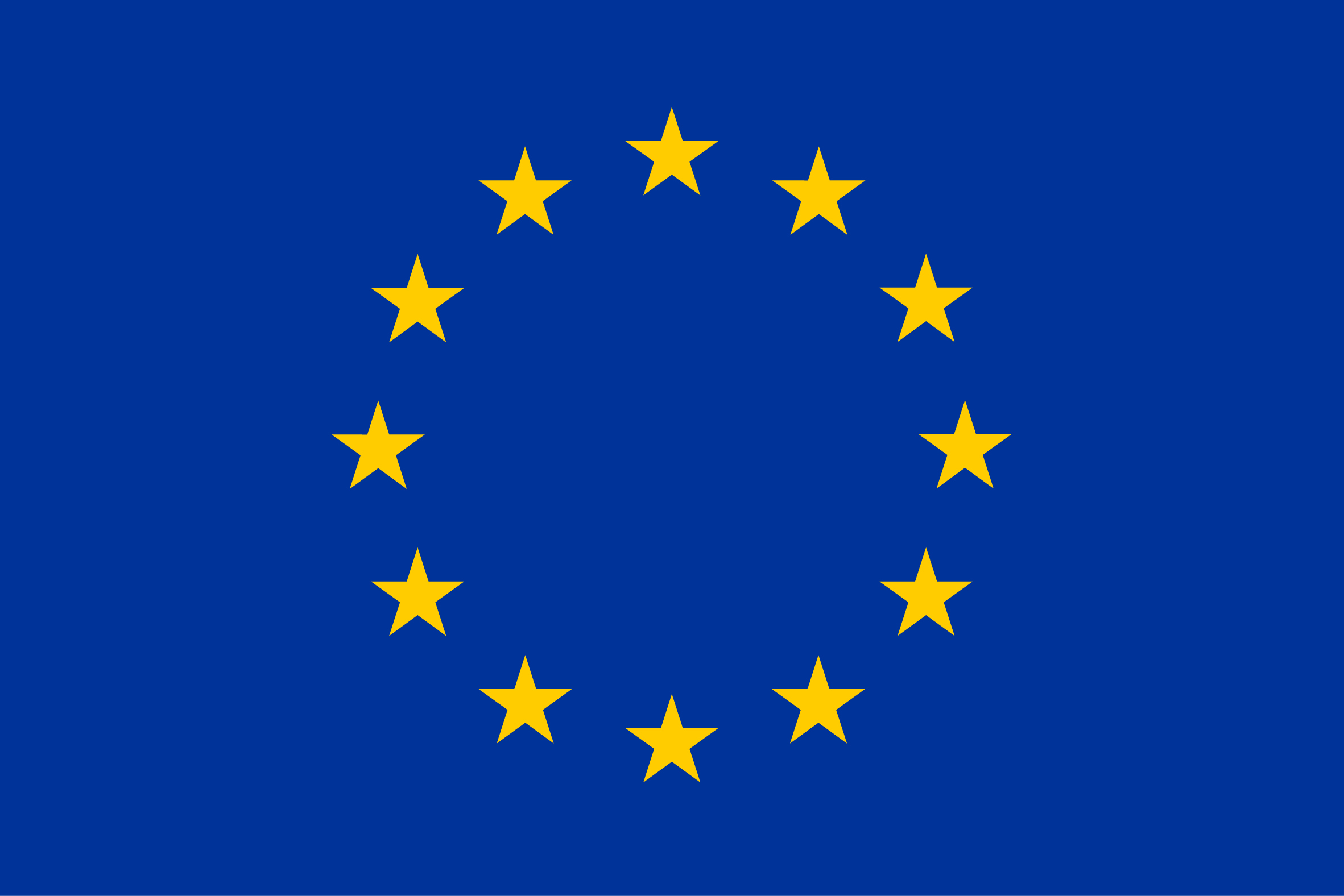}
\end{wrapfigure}  RR is supported by the European Union’s Horizon 2020 research and innovation programme under the Marie Skłodowska-Curie grant agreement no. 945045,  and by the NWO Gravitation project NETWORKS under grant no. 024.002.003.

\section{Main results}
\subsection{The model}
For any two vertices $x,y\in\T_m$, we denote $d(x,y)$ to be the graph distance between $x$ and $y$. Then the Laplacian, whose definition was given in~\eqref{eq:def_Delta}, can also be viewed as the following matrix:
\begin{align}\label{eq:matrix_lap}
\Delta(x,y) = \begin{cases} -1\hspace{4pt}\text{ if } x=y,\\
\frac1{m} \hspace{10pt}\text{ if } d(x,y)=1,\\
0\hspace{12pt}\text{ otherwise.}\end{cases}
\end{align}
We write $\Delta^2$ for its iteration, i.e., $\Delta^2f_x:=\Delta(\Delta f(x))$ and define $\Delta^2_\Lambda$ to be the matrix $(\Delta^2(x,y))_{x,y\in\Lambda}$.
\begin{lemma}\label{lem:model_def}
The Gibbs measure $\prob_\Lambda$ on $\R^\Lambda$ with $0$-boundary conditions outside $\Lambda$ given by~\eqref{eq:prob_measure} exists for any finite subset $\Lambda$. It is the centered Gaussian field on $\Lambda$ with covariance matrix $(\Delta^2_\Lambda)^{-1}$. 
\end{lemma}
\begin{proof}
We first prove that $\Delta^2$ is symmetric and positive definite, i.e., for any function $f:\T_m \to \R$ which vanishes outside a finite subset and which is not identically zero
\begin{align}\label{eq:pd_bilap}
\sum_{x,y\in\T_m} f(x) \Delta^2(x,y) f(y) >0.\end{align}
From~\eqref{eq:matrix_lap} it is clear that $\Delta$ is symmetric, and hence $\Delta^2$ is so. Let $g=\Delta f$ and to prove~\eqref{eq:pd_bilap} we observe that
\begin{align*}
\sum_{x,y\in\T_m} f(x) \Delta^2(x,y) f(y)&= \sum_{x\in\T_m} f(x) \Delta g(x)
= \frac{1}{m}\sum_{x\in\T_m} f(x) \sum\limits_{y\sim x} (g(y)-g(x))\\
&= \frac{1}{m}\sum_{x\in\T_m} g(x) \sum\limits_{y\sim x} (f(y)-f(x))=\sum_{x\in\T_m} g(x) g(x) >0.
\end{align*}
Also, one can show using summation by parts that if $\varphi:\T_m\to\R$ vanishes outside $\Lambda$ then  
\begin{align*}
\sum_{x\in\T_m}(\Delta\varphi_{x})^2 =\sum_{x\in\T_m}\varphi_x\Delta^2\varphi_x.
\end{align*}
The proof is now complete by using Proposition 13.13 of~\cite{Georgii}.
\end{proof}

\subsection{Main results}
We denote the root of the tree by $o$. We will consider $m\ge 3$. In the case when $m=2$ the tree is isomorphic to $\Z$ and the MM on $\Z$ has been studied in the literature, see for instance~\cite{CarJDScaling,CaravennaDeuschel_pin}. For any $n\in \N$, we define 
\begin{align*}
V_n:= \{x\in\T_m: d(o,x) \le n\}.
\end{align*}
Let $\vr=(\varphi_{x})_{x\in\T_m}$ be the membrane model on $\T_m$ with zero boundary conditions outside $V_n$. In this case, we denote the corresponding measure $\prob_{V_n}$ by $\prob_n$. Also we denote the covariance function for this model by $G_n$, that is, $G_n(x,y):= \E_n[\vr_x\vr_y]$. Let $(S_k)_{k\ge0}$ be the simple random walk on $\T_m$. We write $\prob_x$ for the canonical law of the simple random walk starting at $x$. The following theorem proves the existence of the infinite volume limit.
\begin{theorem}\label{thm:inf_vol}
The measures $\prob_n$ converge weakly to a measure $\prob$, which is the law of a Gaussian process $(\vr_{x})_{x\in\T_m}$ with covariance function $G$ given by
\begin{align*}
G(x,y):=\E[\vr_x\vr_y] = \E_x\left[\sum_{k=0}^\infty (k+1)\one_{[S_k=y]}\right] = \sum_{k=0}^\infty (k+1)\prob_x(S_k=y).
\end{align*}
\end{theorem}
We will see later (in Lemma~\ref{lem:G}) that for any $x\in \T_m$ $$G(x,x)=G(o,o)=\frac{(m-1)((m-1)^2+1)}{(m-2)^3}.$$
We define two sequences as follows
\begin{align*}
b_n:= \sqrt{G(o,o)}\left[\sqrt{2\log N} -\frac{\log\log N + \log{(4\pi)}}{2\sqrt{2\log N}}\right],\,\,\,a_n:= G(o,o)b_n^{-1},
\end{align*}
where $N:=|V_n|$. We have 
\begin{align}
N= 1+\sum_{k=1}^n m(m-1)^{k-1} = \frac{m(m-1)^n-2}{m-2}.\label{eq:relNn}
\end{align}
Our main result in this paper concerns the scaling limit of the maximum of the field, namely the Gumbel convergence of the rescaled maximum.
\begin{theorem}\label{thm:max}
For any $\theta\in\R$
\begin{align*}
\lim_{n\to\infty} \prob \left( \frac{\max_{x\in V_n} \vr_x - b_n}{a_n} \le \theta\right) = \exp(-\e^{-\theta}).
\end{align*}
\end{theorem}

We show in the following result that up to the first order the constants do not change for the extremes and when we look at the expected maximum under the finite volume, it still converges to $\sqrt{G(o,o)}$, after appropriate scaling. The same result can be proved under the infinite volume measure, so we stick to the finite volume case, the infinite volume situation being completely analogous.

\begin{theorem}\label{thm:exp_max_0}
For $m\ge 3$,
\begin{align*}
\lim_{n\to\infty} \frac{\E_n \left[ \max_{x\in V_n} \varphi_x \right]}{\sqrt{2\log N}}= \sqrt{G(o,o)}.
\end{align*}
\end{theorem}

In case of the finite volume field we show that the maximum field normalised to have variance one converges in distribution to the Gumbel distribution. We define $B_n:= b_n/\sqrt{G(o,o)}$ and $A_n:= B_n^{-1}$. 
\begin{theorem}\label{thm:max_0}
Let $\psi_x= \varphi_x/\sqrt{G_n(x,x)}$ for $x\in V_n$. Then for any $\theta\in\R$ and $m\ge 14$ we have that
\begin{align*}
\lim_{n\to\infty} \prob_n \left( \frac{\max_{x\in V_n} \psi_x - B_n}{A_n} \le \theta\right) = \exp(-\e^{-\theta}).
\end{align*}
\end{theorem}

\begin{remark}
In exposing our results we keep all the constants depending on $m$ explicit. We emphasize that the bound $m\ge 14$ need not be optimum as the constants are not known to be sharp.
\end{remark}

\subsection{Discussion}
\begin{itemize}[leftmargin=*] 
\item Our result is heuristically motivated by the fast decay of correlations of the DGFF on a tree. As we shall see in Lemma~\ref{lem:G} correlations decay exponentially in the distance between points, which suggests a strong decoupling and a behavior of the rescaled maximum similar to that of independent and identically distributed Gaussians. 
\item In Theorem~\ref{thm:max}, the scaling constants show that the correlation structure can be ignored for the extremes and the behaviour is similar to that of i.i.d.~centered Gaussian random variables with variance $G(o,o)$. In the finite volume case, we rescaled the field to have variance one and hence the behaviour remains the same as that of the i.i.d.~case. 
An interesting open problem is whether this behavior is retained for the finite volume field divided by $\sqrt{G(o,\,o)}$. This convergence is by no means trivial to obtain as the finite volume variance convergence to $G(\cdot,\,\cdot)$ is not uniform, in particular the error depends on the distance to the leaves. Moreover, since the size of the boundary of a finite tree is non-negligible with respect to the total size we cannot claim that extremes are achieved in the bulk and not near the boundary. 

\item For scaling extremes we use a comparison theorem of \cite{hj90} that is based on Stein's method. There are many different approaches to the question of convergence of extremes using Stein's method, one notable instance being~\cite{AGG}. Compared to their method the advantage of~\cite{hj90} is that it does not require to control the conditional expectation of the field that emerges from the spatial Markov property. While for other interfaces, like the DGFF, the harmonic extension has a closed form in terms of random walk probabilities, the biharmonic extension is more subtle to bound, and~\cite{hj90} allow one to bypass this step.
\item The main contribution of the article is the analysis of the covariance structure for a membrane model on the tree. We are not aware of any prior work which deals with the bilaplacian model on graphs beyond $\Z^d$, whilst there is an extensive literature on the discrete Gaussian free field on general graphs. We exploit the representation of the solution of a biharmonic Dirichlet problem in terms of the random walk on the graph. In the bulk the behaviour is easy to derive and is close to that of $$\overline{G}_n(x,y):= \E_x\left[\sum\limits_{k=0}^{\tau_{0}-1} (k+1)\one_{[S_k=y]} \right],$$
where $\tau_0$ is the first exit time from a bounded subgraph (see Section~\ref{section:rwrep}). Around the boundary additional effects arising out of boundary excursions kick in, in particular we will use the successive excursion times of the random walk and the local time of the random walk between two consecutive visits to the complement of a set. Control of such observables on general graphs will open up avenues for further interesting studies in the area of random interfaces. 

\item Although we consider regular trees we believe that our results can be extended to rooted trees where the same scaling limit for the maximum will hold. The case of Galton--Watson tree will be more challenging due to the randomness of the offspring distribution, but would be an intriguing direction to extend the study of random interfaces to random graphs.

\end{itemize}

\paragraph{Structure of the article}

In Section~\ref{section:rwrep} we recall the random walk representation for the solution of the biharmonic Dirichlet problem for a general graph and also rewrite the formula in our set-up. In Section~\ref{section:proofthm:inf_vol} we show that the infinite volume membrane measure exists and provide a proof of Theorem~\ref{thm:inf_vol}. In Section~\ref{section:thm:max} we prove Theorem~\ref{thm:max} providing a limit for the expected maximum under the finite volume measure. In Section~\ref{section:proofthm:exp_max_0} we use the estimates to determine the fluctuations of the extremes in the infinite volume and prove Theorem \ref{thm:exp_max_0}. In Section~\ref{section:thm:max_0} we show the fluctuations of the maximum under the finite volume measure. Section \ref{section:lem:E_n_finer} is devoted to the proof of Lemma~\ref{lem:E_n_finer} which is related to finer estimates on the covariance of the model. \par

\paragraph{{Notation}} In the following $C$ is a generic constant which may depend on $m$ and may change in each appearance within the same equation.
\section{A random walk representation for the covariance function}\label{section:rwrep}
In this section we shall revisit the random walk representation for the covariance function $G_n$. From the definition of the model it follows that $G_n$ satisfies the following Dirichlet problem: for $x\in V_n$
\begin{equation}\label{eq:G_n}
\begin{cases}
\Delta^2 G_n(x,y) =\hspace{5pt}\delta_x(y), & \text{if } y\in V_n\\
G_n(x,y) = \hspace{5pt}0, & \text{if }y\in V_n^c.
\end{cases}
\end{equation}
If one considers the Dirichlet problem above but with $-\Delta$ replacing $\Delta^2$ then the solution is the well-known expected local time of the simple random walk on the graph~\cite[Chapter 1]{ASS}. In our set-up such a general easy formulation is not available. In particular, to the best of the authors' knowledge one cannot relate the covariance of the MM to a stochastic process. The solution is then given by a weighted local times and an expression involving the boundary excursion times of the random walk. The boundary effects are more profound in the membrane model and this is  documented in the existing works on $\Z^d$ (\cite{Kurt_d4, mm_scaling, Mueller:Sch:2017, schweiger:2019}).
\subsection{Intermezzo: random walk representation on general graphs}
In this subsection we discuss the probabilistic solution of the Dirichlet problem for the discrete biharmonic operator obtained by \cite{vanderbei}, whose set-up is much more general in that it considers general graphs and not only trees. We recall it here for completeness. Let $\mathcal G$ be a connected graph and let $\Lambda\subset\mathcal G$ be a finite subgraph.  With a slight abuse we will confound the graph $\mathcal G$ resp. $\Lambda$ with its vertex set, but this should not cause any confusion. Let $\rho$ be a strictly positive measure on the discrete state space $\mathcal G$ and for all $x,y\in \mathcal G$, $ q(x,y)$ be a positive symmetric transition function such that
\begin{equation*}\label{q property}
\sum\limits_{y\in \mathcal G}q(x,y)\rho(y) = 1.
\end{equation*}
Let $P=(p(x,y))_{x,y\in\mathcal G}$ be a transition matrix such that
\begin{equation*}\label{p definition}
p(x,y)=q(x,y)\rho(y).
\end{equation*}
Let $(S_k)_{k\ge 1}$ be a random walk on $\mathcal G$, defined on a probability space $(\Omega,\mathcal{F})$, with transition matrix $P$ making the random walk symmetric. Now the Laplacian operator acting on a function $f:\mathcal G\to \R$ is defined as
\begin{align*}
(\Delta f)(x) = \sum\limits_{y\sim x} p(x,y)(f(y)-f(x)).
\end{align*}
The one-step {transition operator} $P$ is defined as
\begin{align*}\label{transition operator}
(Pf)(x)= \E_x[f(S_1)] = \sum\limits_{y\sim x} p(x,y)f(y).
\end{align*}
Then
\begin{equation*}\label{Laplacian}
\Delta = P-I,
\end{equation*}
where $I$ is the identity operator.
We say that $f$ is a solution to the non-homogeneous Dirichlet problem for the bilaplacian if $f$ satisfies the following: 
\begin{equation}\label{eq:nh_bilap_prob}
\begin{cases}
\Delta^2f(y) = \hspace{5pt}\psi(y), & \text{if } y\in \Lambda\\
f(y) = \hspace{5pt}\phi(y), & \text{if }y\in \partial_2 \Lambda,
\end{cases}
\end{equation}
where $\partial_k\Lambda$ is defined by
\begin{align}
\partial_k\Lambda:= \{z\in\Lambda^{c}: d(z,\Lambda) \le k\},\hspace{10pt} k\ge 1\end{align}
and where $\psi,\,\phi$ are graph functions representing the input datum resp. boundary datum. We want to obtain a probabilistic solution of the problem~\eqref{eq:nh_bilap_prob}. We define $\tau_i$ to be the $(i+1)$-th visit time to $\Lambda^c$ by the random walk $S_k$. Formally,
\begin{equation}\label{boundary_hitting_time}
\tau_i:=\inf\{ k>\tau_{i-1}:S_k\in \Lambda^c \},\hspace{10pt}\tau_{-1}:=-1.
\end{equation}
Note that $\tau_0$ is the first exit time from $\Lambda$. We will keep two assumptions throughout the Section:
\begin{enumerate}[label=(\roman{*}), ref=(\roman{*})]
	\item\label{(i)} $\Lambda\cup\partial_1\Lambda$ is finite;
	\item\label{(ii)} $\E_x\left[\tau_0^2\right]<\infty$ for all $x\in \mathcal G$.
\end{enumerate}

Let $$L^2(\mathcal G,\rho)=\left\{ f\left| \hspace{3pt} f:\mathcal G\to\R \hspace{3pt}\text{such that}\hspace{3pt} \sum\limits_{x\in \mathcal G}f(x)^2\rho(x)<\infty \right.\right\}$$ and the inner-product in $L^2(\mathcal G,\rho)$ be defined as follows: for $f,g\in L^2(\mathcal G,\rho)$
\[
\la f,g \ra_{\mathcal G} := \sum\limits_{x\in \mathcal G}f(x)g(x)\rho(x).
\]
One can show that $P$ is a self-adjoint operator on $L^2(\mathcal G,\rho)$. Hence $\Delta$ is also self-adjoint.
We define
\[
\begin{cases}
M_{-1}:=1 & \\
M_j := \prod\limits_{i=0}^j (\tau_i-\tau_{i-1}-1), & j\ge 0.
\end{cases}
\]
Next we define an operator acting on $\mathcal{A}=\left\{ f\left| f:\Lambda^c\to\R  \right.\right\}$. The operator $Q$ acting on $\mathcal{A}$ is defined as 
\begin{equation}\label{Q operator definition}
(Qf)(x) = \E_x\left[(\tau_1-1)f\left(S_{\tau_1}\right)\right], \hspace{5pt} x\in \Lambda^c.
\end{equation}
Observe that, if $x\in \left( \Lambda\cup\partial_1\Lambda \right)^c$, then $Qf(x)=0$ for all $f\in\mathcal{A}$. Therefore 
\begin{equation*}\label{range of Q}
\text{Range}(Q) \subset \left\{ g\left| g:\Lambda^c\to\R \hspace{5pt}\text{and}\hspace{5pt} g(x)=0,\forall x\notin \partial_1\Lambda  \right.\right\}\subset \mathcal{A}.
\end{equation*}
Since $\partial_1\Lambda$ is finite,  one can show with the help of~\ref{(ii)} that $(Qf)(x)$ is bounded. It can be shown that the operator $Q$ is positive semi-definite on $L^2(\Lambda^c,\rho)$ (see~\cite{vanderbei}). Therefore $Q$ can be diagonalized and can be written as
\begin{equation*}\label{spectral decomposition of Q}
Q = \sum\limits_{\lambda} \lambda\Pi_\lambda
\end{equation*}
where the sum is over all the eigenvalues of $Q$ and $\Pi_\lambda$  are the projection operators onto the eigenspace corresponding to the eigenvalue $\lambda$. Observing the range of the Q-operator in \eqref{range of Q} and the fact that $\partial_1\Lambda$ is finite, we can say that the operator $Q$ is compact and $\text{Range}(Q)$ is finite-dimensional. Therefore, the spectrum of $Q$ is finite. Also, as $Q$ is positive semi-definite, we conclude that all the eigenvalues are non-negative. A probabilistic solution of the problem~\eqref{eq:nh_bilap_prob} is given in~\citet{vanderbei}.
\begin{theorem}[{\citet[{Theorem 4}]{vanderbei}}]\label{thm:the solution}
	Let $(\eta_t)_{t\ge 0}$ be a Poisson process with parameter $1$ which is independent of the random walk $(S_k)$. Then the solution of~\eqref{eq:nh_bilap_prob} is given by
	\begin{equation}\label{eq:nh_solution}
	f(x) = \lim\limits_{t\to\infty} \E_x\left[ \sum\limits_{j=0}^{\eta_t} (-1)^j M_{j-1}\left[ (\tau_j-\tau_{j-1})\phi(S_{\tau_j})+ \sum\limits_{k=\tau_{j-1}}^{\tau_j-1}(k-\tau_{j-1})\psi(S_k) \right] \right].
	\end{equation}
Alternatively, the above solution can be written in terms of the eigenvalues of $Q$ and the corresponding projection operators as follows
\begin{align}\label{eq:nh_solution_2}
f(x)&= \E_x\left[\phi\left(S_{\tau_0}\right)\right] + \sum\limits_{\lambda} \frac{1}{1+\lambda} \E_x\left[ \tau_0\Pi_\lambda(I- \widetilde P)\phi(S_{\tau_0}) \right]\nonumber\\
&\qquad\qquad +\E_x\left[ \sum\limits_{k=0}^{\tau_0-1}(k+1)\psi(S_k) \right] -\sum\limits_{\lambda}\frac{1}{1+\lambda}\E_x \left[ \tau_0\Pi_\lambda h\left(S_{\tau_{0}}\right) \right],
\end{align}
where $\widetilde Pf(z):=\E_z[f(S_{\tau_1})]$ is the operator acting on functions defined on $\partial_1\Lambda$ and
\begin{align*}
h(z) = \E_z\left[ \sum\limits_{k=0}^{\tau_1-1}k\psi(S_k) \right],\quad z\in\Lambda^c.
\end{align*}
\end{theorem}
Note that~\eqref{eq:nh_solution_2} is a re-writing of the solution~\eqref{eq:nh_solution} which is a by-product of Vanderbei's proof.
\begin{remark}
 Without the presence of $\eta$ the series describing the covariances might not be absolutely summable on every graph, as discussed in an example in~\cite{vanderbei}. However in our case, that is for regular trees where we have exponential decay of correlations for $\vr$, one can show that $\eta_t$ does not play any role and can in fact be avoided altogether. Also note that since $\{\tau_i-\tau_{i-1}:\, i\ge 1\}$ need not be i.i.d. the terms involving excursion times to $\Lambda^c$ in~\eqref{eq:nh_solution} must be dealt with care. 

We also note that the representation \eqref{eq:nh_solution_2} is not directly stated as a theorem in \cite{vanderbei} but if one goes through the proof of Theorem 4 in \cite{vanderbei} then it follows immediately.
\end{remark}

\subsection{Back to regular trees}
 In our set-up, $V_n$ consists in the first $n$ generations of the regular tree. Note that $V_n \cup \partial_1 V_n$ is finite. It follows from Lemma~\ref{lem:bound_tau^2} that $\E_x[\tau_0^2]<\infty$ for all $x\in \mathcal \T_m$ with $m\ge 3$,  so that \ref{(i)}-\ref{(ii)} are satisfied. It can be easily proved using the theory of electrical networks that the simple random walk on $\T_m$ is transient for all $m\ge 3$. 
Using the solution~\eqref{eq:nh_solution}  we have the random walk representation of $G_n(x,y)$ as follows:
\begin{equation}
G_n(x,y) = \lim\limits_{t\to\infty}\E_x\left[ \sum\limits_{j=0}^{\eta_t}(-1)^{j}M_{j-1}\sum\limits_{k=\tau_{j-1}}^{\tau_j-1}(k-\tau_{j-1})\one_{[S_k=y]} \right].
\end{equation}
We have used $\phi(z)=0$ and $\psi(z)=\one_{[z=y]}$ in equation~\eqref{eq:nh_solution}. 

We write $G_n(x,y)$ as
\begin{equation}\label{eq:G_n split}
G_n(x,y) = \overline{G}_n(x,y) - E_n(x,y),
\end{equation}
where 
\begin{align}
&\overline{G}_n(x,y):= \E_x\left[\sum\limits_{k=0}^{\tau_{0}-1} (k+1)\one_{[S_k=y]} \right],\nonumber\\
&E_n(x,y):=\lim\limits_{t\to\infty}\E_x\left[ \sum\limits_{j=1}^{\eta_t}(-1)^{j-1}M_{j-1}\sum\limits_{k=\tau_{j-1}}^{\tau_j-1}(k-\tau_{j-1})\one_{[S_k=y]} \right]\label{error:rep}.
\end{align}

Note that $\overline{G}_n(x,y)= \E_{x,y}\left[ \sum_{k=0}^{\tau_{0}-1}\sum_{\ell=0}^{\tau^\prime_{0}-1}  \one_{[S_k= S^\prime_\ell]}\right]$ where $S_k$ and $S^\prime_\ell$ are two independent simple random walks starting from $x$ and $y$ respectively, and $\tau_0$ and $\tau_0^\prime$ are their first visit times to $V_n^c$ respectively. $\overline{G}_n(x,y)$ plays crucial role in the study of the membrane model: in the $\Z^d$ case, it was shown in \cite{Kurt_thesis} that $G_n$ and $\overline{G}_n$ are close in the bulk of the domain. We will also see here that $E_n(x,y)$ plays a role of the error term.  We observe from~\eqref{eq:G_n split} and~\eqref{eq:nh_solution_2} that
\begin{equation}\label{error:2}
E_n(x,y)= \E_x\left[ \tau_0 \sum_{\lambda} \frac{1}{1+\lambda}\Pi_\lambda h(S_{\tau_0})\right],
\end{equation}
where $h(z)=\E_z\left[\sum\limits_{k=0}^{\tau_{1}-1} k\one_{[S_k=y]} \right]$ for $z\in V_n^c$.

\section{Proof of Theorem~\ref{thm:inf_vol}}\label{section:proofthm:inf_vol}
If the infinite volume limit exists then it is supposed to have the covariance function $G$. We first show $G(x,y)= \sum_{k=0}^\infty (k+1) \prob_x(S_k=y)$ can be computed in terms of $d(x,y)$ for a $m$-regular tree and that it has exponential decay in the distance $d(x,y)$. 
\begin{lemma}\label{lem:G}
We have for any $x, y\in \T_m$ that
\begin{align*}
G(x,y)= \frac{(d(x,y)+1) m(m-1)(m-2) +2(m-1)}{(m-2)^3(m-1)^{d(x,y)}}.
\end{align*}
\end{lemma}
\begin{proof}
We define the Green's function of the simple random walk on $\T_m$ as the power series
\begin{align*}
\Gamma(x,y|\z):=\sum_{k=0}^\infty \prob_x(S_k=y)\z^k,\,\,x,\,y\in\T_m,\quad\z\in\mathbb{C}.
\end{align*}
From \citet[Lemma 1.24]{Woess:book} we have
\begin{align}\label{eq:lem_1.24}
\Gamma(x,y|\z) = \frac{2(m-1)}{m-2+\sqrt{m^2-4(m-1)\z^2}}\left( \frac{m-\sqrt{m^2-4(m-1)\z^2}}{2(m-1)\z}\right)^{d(x,y)}.
\end{align}
We fix $x,y\in\T_m$ and write $d=d(x,y)$, $g(\z):=\Gamma(x,y|\z)$ for $\z\in\mathbb{C}$.
Now observe that
\begin{align*}
G(x,y)= g'(1) + g(1).
\end{align*}
From~\eqref{eq:lem_1.24} we get
\begin{align*}
&\log(g(\z))= \log(2(m-1))-\log\left(m-2+\sqrt{m^2-4(m-1)\z^2}\right) \\
&\qquad\qquad\qquad\qquad+ d\log\left(m-\sqrt{m^2-4(m-1)\z^2}\right)
-d\log(2(m-1)) -d\log \z.
\end{align*}
So taking a derivative we have
\begin{align*}
&\frac{g{'}(\z)}{g(\z)}= \frac{8(m-1)\z}{2\left(m-2+\sqrt{m^2-4(m-1)\z^2}\right)\sqrt{m^2-4(m-1)\z^2}}\\
&\qquad\qquad\qquad\qquad\qquad\qquad + \frac{8d(m-1)\z}{2\left(m-\sqrt{m^2-4(m-1)\z^2}\right)\sqrt{m^2-4(m-1)\z^2}} - \frac{d}{\z}
\end{align*}
and hence evaluation at $\z=1$ gives
\[
\frac{g{'}(1)}{g(1)}= \frac{2(m-1)}{(m-2)^2} + \frac{2d(m-1)}{m-2}-d= \frac{2(m-1)+dm(m-2)}{(m-2)^2}.
\]
Also 
\begin{align}
g(1)= \frac{1}{(m-2)(m-1)^{d-1}}.\label{greensrw:tree}
\end{align}
Now we obtain
\begin{align*}
G(x,y)&= g{'}(1) + g(1)
=g(1)\left(\frac{g{'}(1)}{g(1)}+1\right)\\
&=\frac{2(m-1)+dm(m-2)+(m-2)^2}{(m-2)^3(m-1)^{d-1}}=\frac{(d+1)m(m-2)+2}{(m-2)^3(m-1)^{d-1}}.
\end{align*}
\end{proof}
The behavior of $G$ depends crucially on the graph distance $d(x,y)$ between two points $x$ and $y$ on the tree. We would need an estimate on the number of points $(x,y)$ which are at a fixed distance $k$. The following lemma gives a bound on this.
\begin{lemma}\label{lem:C_k}
	Let $$C_k:=|\{(x,y)\in V_n\times V_n: d(x,y)=k\}|.$$
Then 
$$C_k \le C (m-1)^{n+\floor{\frac{k}{2}}},$$
where $C$ is a constant which depends on $m$.		
\end{lemma}
\begin{proof}
	Let $e(x):=d(o,x)$ and for any $\ell>0$ define
	$$\partial_0 V_\ell:=\{x\in V_n:e(x)=\ell\}.$$
	In other words, $\partial_0 V_\ell$ is the set of all leaf-vertices in $V_\ell$. Let us now fix any vertex in $\partial_0 V_{n-\ell}$ and count the number of $y$'s in $V_n$ such that $d(x,y)=k$. Notice that for any $x\in \T_m\setminus \{o\}$, we have $e(y)=e(x)\pm 1$ for all $y$ with $x\sim y$. Moreover, $e(y)=e(x)- 1$ holds for only one such $y$ and $e(y)=e(x)+ 1$ holds for remaining $(m-1)$ many such $y$. In other words, from any $x\in \T_m\setminus \{o\}$ there is only one way to move closer to $o$ and $(m-1)$ way to move farther away from $o$.
	Now let us first consider the case when $k\le n$. If $0\le \ell<k$, then for any $x\in\partial_0 V_{n-\ell}$ the number of vertices in $V_n$ which are at a distance $k$ from $x$ is $(m-1)^{\floor{(k-\ell)/2}+\ell}$. This is because the unique path of length $k$ from $x$ to some vertex in $V_n$ must consists of first $\ceil{(k-l)/2}$ steps moving closer to $o$ and the remaining steps moving in any of the $(m-1)$ available directions. The situation is explained graphically in Figure~\ref{fig:lemma9}. On the other hand if $\ell \ge k$, then for any $x\in \partial_0 V_{n-\ell}$  the number of vertices in $V_n$ which are at a distance $k$ from $x$ is $m(m-1)^{k-1}$. Also, note that for any $j\ge 1$ there are $m(m-1)^{j-1}$ vertices in $\partial_0 V_j$. Therefore in this case
	\begin{align*}
	&C_k = \sum_{\ell=0}^{k-1} m(m-1)^{n-\ell-1}(m-1)^{\floor{\frac{k-\ell}{2}}+\ell}\\
	&\qquad\qquad\qquad\qquad + \sum_{\ell=k}^{n-1} m(m-1)^{n-\ell-1}m(m-1)^{k-1} + m(m-1)^{k-1}\\
	&\qquad = m(m-1)^{n+\floor{\frac{k}2}-1}\sum_{\ell=0}^{k-1}(m-1)^{-\floor{\frac{\ell}2}} + m^2 (m-1)^{n-2}\sum_{\ell=0}^{n-k-1}(m-1)^{-\ell} + m(m-1)^{k-1}\\
	&\qquad \le C (m-1)^{n+\floor{\frac{k}{2}}}.	
	\end{align*}
	Now we consider the case when $k>n$. In this case for any $x\in \partial_0 V_{n-\ell}$ the maximum number of vertices which are at a distance $k$ from $x$ is $(m-1)^{\floor{(k-\ell)/2}+\ell}$ (here we are over-counting and for $k>2n-\ell$ this number is $0$). Therefore
	\begin{align*}
	C_k &\le \sum_{\ell=0}^n m(m-1)^{n-\ell-1}(m-1)^{\floor{\frac{k-\ell}2}+\ell}\\
	&= m(m-1)^{n-1+\floor{\frac{k}{2}}}\sum_{\ell=0}^n(m-1)^{-\floor{\frac{\ell}{2}}}\\
	&\le C (m-1)^{n+\floor{\frac{k}{2}}}.
	\end{align*}\hfill\qedhere
\end{proof}
\begin{figure}[ht!]\center
\includegraphics[scale=1]{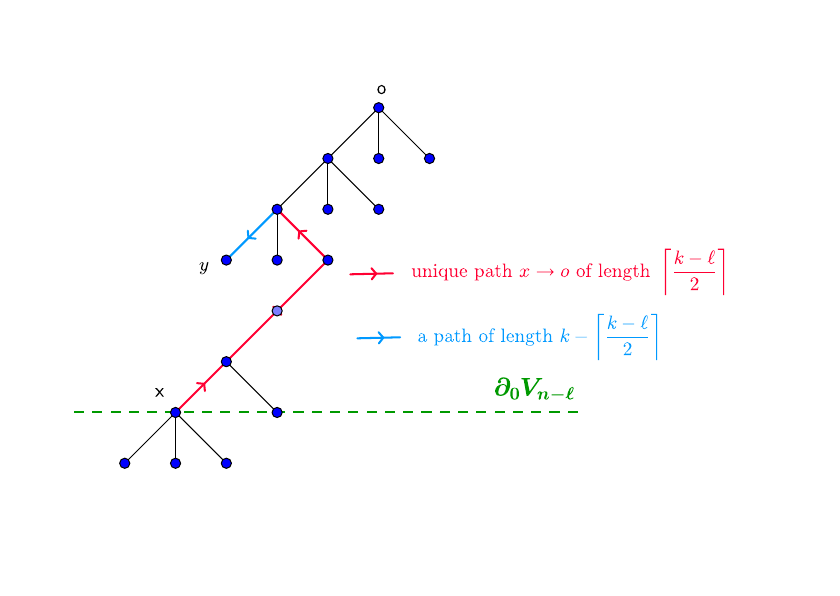}
\caption{Case $0\le \ell<k\le n$ for the proof of Lemma 10. If $d(x,\,y)=k$, in order to reach $y$ from $x$ one must move to their least common ancestor in $\lceil(k-\ell)/2\rceil$ steps and from there move to $y$ in $k-\lceil(k-\ell)/2\rceil$ steps. The $3-$ary tree is partially drawn for simplicity.}
\label{fig:lemma9}
\end{figure}
Our final lemma before the proof of Theorem~\ref{thm:inf_vol} gives an estimate on the second moment of the exit time.
\begin{lemma}\label{lem:bound_tau^2}
 For any $x\in V_n$
\begin{align}\label{eq:bound_tau^2}
\E_x[\tau_0^2] \le C d(x,\partial_1V_n)^2.
\end{align}
\end{lemma}
\begin{proof}
 Observe that
 $\E_x\left[\sum_{z\in V_n} \sum_{\ell=0}^{\tau_0-1}\ell\one_{[S_\ell=z]}\right] = \E_x\left[{\tau_0 (\tau_0-1)}/{2}\right]$ which implies
\[
\E_x\left[\tau_0^2\right] \le 2 \E_x\left[\sum_{z\in V_n} \sum_{\ell=0}^{\tau_0-1}\ell\one_{[S_\ell=z]}\right]\\\le 2 \sum_{z\in V_n} G(x,z).
\]
By using Lemma~\ref{lem:G} we get
\[
\E_x\left[\tau_0^2\right]  \le C \sum_{z\in V_n} d(x,z)(m-1)^{-d(x,z)}
= C \sum_{k=0}^{2n} \sum_{z:d(x,z)=k} k(m-1)^{-k}.
\]
Now suppose $\ell= d(0,x)$. With arguments similar to the proof of Lemma~\ref{lem:C_k}, we have
 $$|\{z: d(x,z)=k\}|\le \begin{cases}
 m(m-1)^k & \text{ if } 0\le k\le n-\ell\\
 (m-1)^{\frac{k}{2} + \frac{n-\ell}{2}} & \text{ if } n-\ell+1\le k\le n+\ell.
\end{cases}$$
Hence we have
\begin{align*}
\E_x[\tau_0^2] & \le C \sum_{k=0}^{n-\ell} m(m-1)^k k(m-1)^{-k} + C \sum_{k= n-\ell+1}^{n+\ell} (m-1)^{\frac{k}{2} + \frac{n-\ell}{2}} k(m-1)^{-k}\\
& \le C\sum_{k=0}^{n-\ell} k + \sum_{k=1}^{2\ell}(k+n-\ell)(m-1)^{-\frac{k}{2}} \le C (n-\ell)^2 \le C d(x,\partial_1V_n)^2.\qedhere
\end{align*}
\end{proof}
\begin{proof}[Proof of Theorem~\ref{thm:inf_vol}]
 Since the random walk on $\T_m$ starting from vertex $x$ is transient, $\tau_0$ is finite almost surely and $\tau_0\ge n-d(o,x)$ for all $n\ge 1$. Therefore $\tau_0$ increases to $\infty$ as $n\to \infty$. Hence as an immediate conclusion $\{\overline{G}_n(x,y)\}_{n\ge1}$ is an increasing sequence. By monotone convergence theorem we have
\begin{align*}
\lim_{n\to\infty}\overline{G}_n (x,y) =\E_x\left( \sum\limits_{k=0}^{\infty}(k+1)\one_{[S_k=y]} \right).
\end{align*} 
We now show that $|E_n(x,y)|\to 0$ as $n$ tends to infinity so that we have $$\lim_{n\to\infty}G_n (x,y) = \E_x\left( \sum\limits_{k=0}^{\infty}(k+1)\one_{[S_k=y]} \right).$$ As the measures under consideration are Gaussian, this will complete the proof.  Recall the representation of $E_n(x,y)$ from \eqref{error:2} with $h(z)=\E_z\left[\sum\limits_{k=0}^{\tau_{1}-1} k\one_{[S_k=y]} \right]$. 
Since $Q$ has eigenvalues $\lambda$ it can be seen that $\left(\frac{1}{I+Q}\right)= \sum_{\lambda} \frac{1}{1+\lambda} \Pi_{\lambda}$. So we can rewrite \eqref{error:2} as
\begin{align*}
E_n(x,y)&=\E_x\left[ \tau_0 \left(\frac{1}{I+Q}\right)h(S_{\tau_0})\right]\\
&=\E_x\left[\sum_{z\in V_n^c} \tau_0 \one_{[S_{\tau_0}=z]} \left(\frac{1}{I+Q}\right)h(z)\right] = \sum_{z\in V_n^c}\E_x\left[\tau_0 \one_{[S_{\tau_0}=z]}\right]\left(\frac{1}{I+Q}\right)h(z).
\end{align*}
By Cauchy-Schwarz inequality and the fact that $\|\left({I+Q}\right)^{-1}\| \le 1$ we get
\begin{align}\label{eq:E_n}
E_n(x,y)^2 &\le \sum_{z\in V_n^c} \left(\E_x\left[\tau_0 \one_{[S_{\tau_0}=z]}\right]\right)^2 \sum_{z\in V_n^c}h(z)^2 \nonumber\\
&= \sum_{z\in \partial_1V_n} \left(\E_x\left[\tau_0 \one_{[S_{\tau_0}=z]}\right]\right)^2 \sum_{z\in \partial_1V_n}h(z)^2 \nonumber\\
& \le \sum_{z\in \partial_1V_n} \E_x[\tau_0^2]\prob_x(S_{\tau_0}=z)\sum_{z\in \partial_1V_n} G(z,y)^2\nonumber\\
& \le \E_x[\tau_0^2]\sum_{z\in \partial_1V_n} G(z,y)^2 .
\end{align}
Now we obtain a bound for the second factor in~\eqref{eq:E_n}. We have for $y\in\partial_0 V_\ell$
\begin{align*}
\sum_{z\in \partial_1V_n} G(z,y)^2 &\le C \sum_{k=n-\ell+1}^{n+\ell+1} \sum_{z:d(z,y)=k} k^2(m-1)^{-2k} \le C \sum_{k=n-\ell+1}^{n+\ell+1} m(m-1)^{k-1}k^2(m-1)^{-2k}\\
& \le C \sum_{k=1}^{2\ell} (k+n-\ell)^2 (m-1)^{-(k+n-\ell)} \le C (n-\ell)^2 (m-1)^{-(n-\ell)}.
\end{align*} 
Thus we have 
\begin{align}\label{eq:E_n_2}
\sum_{z\in \partial_1V_n} G(z,y)^2 \le C d(y, \partial_1V_n)^2 (m-1)^{-d(y, \partial_1V_n)}.
\end{align}
Plugging in the bounds~\eqref{eq:bound_tau^2},~\eqref{eq:E_n_2} in~\eqref{eq:E_n} we obtain
\begin{align*}
E_n(x,y)^2 \le C d(x,\partial_1V_n)^2d(y, \partial_1V_n)^2 (m-1)^{-d(y, \partial_1V_n)}.
\end{align*}
From symmetry we can conclude
\begin{align}\label{eq:E_n_bound}
E_n(x,y)^2 \le C d(x,\partial_1V_n)^2d(y, \partial_1V_n)^2\min\{ (m-1)^{-d(x, \partial_1V_n)}, (m-1)^{-d(y, \partial_1V_n)}\}.
\end{align}
It follows from~\eqref{eq:E_n_bound} that $|E_n(x,y)|\to 0$ as $n\to\infty$.
\end{proof}
An alternative proof of~\eqref{eq:E_n_bound} using the maximum principle for harmonic functions is provided in Appendix~\ref{app:alternative}.
\section{Proof of Theorem~\ref{thm:max}}\label{section:thm:max}
In this section we prove Theorem~\ref{thm:max}.
\begin{proof}[Proof of Theorem~\ref{thm:max}]
	For any $x\in\T_m$, we define $\psi_x:=\vr_x/\sqrt{G(o,o)}$. Then for each $x$, $\psi_x$ is a standard Gaussian random variable.  Recall that for any $n$
\begin{align*}
B_n=\frac{b_n}{\sqrt{G(o,o)}}= \sqrt{2\log N} -\frac{\log\log N + \log{(4\pi)}}{2\sqrt{2\log N}},\,\,\,A_n= B_n^{-1}.
\end{align*}
We set
\begin{align*}
u_n(\theta):= A_n\theta + B_n,\,\,\theta\in\R.
\end{align*}
Let $\theta\in\R$ be fixed and define
\begin{align*}
W_n:=\sum_{x\in V_n}\one_{[\psi_x > u_n(\theta)]},\,\,\lambda_n:=\E[W_n]=\sum_{x\in V_n} \prob(\psi_x>u_n(\theta)).
\end{align*}
Let $Poi(\lambda)$ denote a Poisson random variable with parameter $\lambda$. We shall use a Binomial-to-Poisson approximation by \cite{hj90}. By Theorem 3.1 of~\cite{hj90} we get
\begin{align}\label{eq:thm3.1}
&d_{TV}\left(W_n, Poi(\lambda_n)\right) \le \frac{1-\e^{-\lambda_n}}{\lambda_n}\left[\sum_{x\in V_n} \prob(\psi_x>u_n(\theta))^2\nonumber\right.\\
&\left.\qquad\qquad\qquad\qquad\qquad\qquad + \sum_{x,y\in V_n, x\neq y}|\mathbf{Cov}\left(\one_{[\psi_x > u_n(\theta)]}, \one_{[\psi_y > u_n(\theta)]}\right)|\right],
\end{align}
where $d_{TV}$ is the total variation distance. We want to prove that $d_{TV}\left(W_n, Poi(\lambda_n)\right)$ goes to zero as $n$ tends to infinity. Once we prove it, we will have
$$|\prob({W_n=0})- \e^{-\lambda_n}| \to 0.$$
But 
\begin{align*}\prob({W_n=0})=P\left(\max_{x\in V_n} \psi_x \le u_n(\theta)\right)=P\left( \frac{\max_{x\in V_n} \vr_x - b_n}{a_n} \le \theta\right).
\end{align*} 
 Using Mill's ratio
 \begin{align*} \left(1-\frac1{t^2} \right)\frac{\e^{-\frac{t^2}2}}{\sqrt{2\pi}t} \le P\left(\mathcal N(0,1) >t\right) \le  \frac{\e^{-\frac{t^2}2}}{\sqrt{2\pi}t} \qquad t>0, \end{align*}
  one can show that $\lambda_n = N\prob(\psi_o>u_n(\theta))$ converges to $\e^{-\theta}$ as $n$ tends to infinity. Hence the proof will be complete.

We now obtain bounds for the terms in~\eqref{eq:thm3.1} to prove that $d_{TV}\left(W_n, Poi(\lambda_n)\right)$ goes to zero as $n$ tends to infinity.
Another use of Mill's ratio gives
\begin{align*}
\frac{1-\e^{-\lambda_n}}{\lambda_n}\sum_{x\in V_n} \prob(\psi_x>u_n(\theta))^2 \to 0.
\end{align*}
Now we give a bound on the other term in~\eqref{eq:thm3.1}. Let $x,y\in V_n$ with $d(x,y)=k\ge 1$ and define $r_k:=\mathbf{Cov}(\psi_x, \psi_y)$. From Lemma~\ref{lem:G} it follows that $r_k$ depends only on $k$, not on $x$ or $y$, and moreover $0<r_k <1$ for all $k\ge 1$. Now from Lemma 3.4 of~\cite{hj90} we obtain the following bounds:
\begin{align}
&0\le \mathbf{Cov}\left(\one_{[\psi_x > u_n(\theta)]}, \one_{[\psi_y > u_n(\theta)]}\right) \le C N^{-\frac{2}{1+r_k}} \label{eq:bound1}\\
& 0\le \mathbf{Cov}\left(\one_{[\psi_x > u_n(\theta)]}, \one_{[\psi_y > u_n(\theta)]}\right) \le Cr_k N^{-2}\log N \,\e^{2r_k\log N}.\label{eq:bound2}
\end{align}
We fix $\delta$ such that $0\le \delta < (1-r_1)/(1+r_1)$. We have 
\begin{align*}
I_n:&=\frac{1-\e^{-\lambda_n}}{\lambda_n}\sum_{x,y\in V_n, x\neq y}\left|\mathbf{Cov}\left(\one_{[\psi_x > u_n(\theta)]}, \one_{[\psi_y > u_n(\theta)]}\right)\right|\\
& \le C \sum_{k=1}^{2n} \sum_{x,y\in V_n,\, d(x,y)=k}\mathbf{Cov}\left(\one_{[\psi_x > u_n(\theta)]}, \one_{[\psi_y > u_n(\theta)]}\right).
\end{align*}
Now we use the bounds~\eqref{eq:bound1}, ~\eqref{eq:bound2} and get the following bound for the above term:
\begin{align*}
I_n &\le C \sum_{k=1}^{\floor{2n\delta}} C_k N^{-\frac{2}{1+r_k}} + C \sum_{\floor{2n\delta}+1}^{2n} C_k N^{-2}\log N \e^{2r_k\log N}\\
& \le C \sum_{k=1}^{\floor{2n\delta}}(m-1)^{n+\floor{\frac{k}2}}N^{-\frac{2}{1+r_k}} + C\sum_{k=\floor{2n\delta}+1}^{2n} (m-1)^{n+\floor{\frac{k}2}} k(m-1)^{-k} N^{-2}\log N \e^{2r_{\floor{2n\delta}}\log N}\\
& \le C (m-1)^{n+n\delta-\frac{2n}{1+r_1}} + C n^3 (m-1)^{-n+2nr_{\floor{2n\delta}}}.
\end{align*}
Note that in the second inequality we have used Lemma~\ref{lem:C_k} and the fact that $r_k$ is decreasing in $k$ with $r_k \le C k(m-1)^{-k}$, which follows from Lemma~\ref{lem:G}. Now observe that by definition $1+\delta < 2/(1+r_1)$ and $r_{\floor{2n\delta}} < 1/2$ for large enough $n$. Thus $I_n$ goes to $0$ as $n$ tends to infinity and we conclude $d_{TV}\left(W_n, Poi(\lambda_n)\right)$ goes to zero as $n$ tends to infinity. 
\end{proof}

\section{Proof of Theorem~\ref{thm:exp_max_0}}\label{section:proofthm:exp_max_0}
In this section we prove Theorem~ \ref{thm:exp_max_0}. We use the following 
\begin{lemma}\label{lem:G_(x,x)_upper_bound}
For any $x\in V_n$ one has $G_n(x,x) \le G(o,\,o)$.\end{lemma}
\begin{proof}
The proof can be readily adapted from that of~\citet[Corollary 3.2]{BCK} which is carried out for $\Z^d$.
\end{proof}
The following lemma gives a bound on $\E_x[\tau_0]$ in terms of the distance of the point $x$ from the boundary.
\begin{lemma} \label{lemma:bound_tau}
For $x\in V_n$
\begin{align}\label{eq:bound_tau}
\E_x[\tau_0] \le C_1(m) d(x,\partial_1V_n),
\end{align}
where 
\begin{equation}\label{eq:C1}
C_1(m)=\frac{(m-1)^{\frac32}}{(m-2)(\sqrt{m-1}-1)} +\frac{m}{m-2}.
\end{equation}
\end{lemma}
\begin{proof}
Using \eqref{greensrw:tree} we have
\begin{align*}
\E_x[\tau_0]&= \E_x\left[\sum_{y\in V_n} \sum_{\ell=0}^{\tau_0-1}\one_{[S_\ell=y]}\right]\le \sum_{y\in V_n} \E_x\left[\sum_{\ell=0}^{\infty}\one_{[S_\ell=y]}\right]\\
&= \sum_{y\in V_n} \frac{1}{(m-2)(m-1)^{d(x,y)-1}}.
\end{align*}
Suppose $d(o,x)=k$. Then similarly to the proof of Lemma~\ref{lem:C_k} we argue that 
\begin{align*}
|\{y\in V_n: d(x,y)=\ell\}|=\begin{cases} m(m-1)^{\ell-1},\,\,\,\text{ if } 1\le \ell \le n-k\\
(m-1)^{\floor{\frac{l-n+k}2} +(n-k)},\,\,\,\text{ if } n-k < \ell \le n+k\\
0,\,\,\,\text{ if } \ell > n+k \end{cases}.
\end{align*}
Therefore splitting the sum according to the distance $d(x,y)$ we get,
\begin{align*}
\E_x[\tau_0] &= \frac{m-1}{m-2}  \sum_{y\in V_n} (m-1)^{-d(x,y)}\\ 
& \le  \frac{m-1}{m-2} \left[ 1+ \sum_{\ell=1}^{n-k} \sum_{y\in V_n: d(x,y)=\ell} (m-1)^{-\ell} + \sum_{\ell=n-k+1}^{n+k} \sum_{y\in V_n: d(x,y)=\ell} (m-1)^{-\ell} \right]\\
& = \frac{m-1}{m-2} \left[ 1+ \sum_{\ell=1}^{n-k} m(m-1)^{\ell-1}(m-1)^{-\ell} + \sum_{\ell=n-k+1}^{n+k} (m-1)^{\floor{\frac{\ell-n+k}2} +(n-k)} (m-1)^{-\ell} \right].
\end{align*}
As $\floor{\frac{\ell-n+k}2} \le \frac{\ell}2-\frac{n-k}2$, we have 
\begin{align*}
\E_x[\tau_0] & \le \frac{m-1}{m-2} \left[ 1+ \frac{m}{m-1} (n-k) + \sum_{\ell=n-k+1}^{n+k} (m-1)^{\frac{\ell}2 -\frac{n-k}2 +n-k-\ell}\right]\\
& =\frac{m-1}{m-2} \left[ 1+ \frac{m}{m-1} (n-k) + \sum_{\ell=1}^{2k} (m-1)^{-\frac{\ell}2}\right]\\
& = \frac{m-1}{m-2} \left[ 1+ \frac{m}{m-1} (n-k) + (m-1)^{-\frac12} \frac{1-(m-1)^{-k}}{1-(m-1)^{-\frac12}} \right]\\
&\le \left( \frac{(m-1)^{\frac32}}{(m-2)(\sqrt{m-1}-1)} +\frac{m}{m-2} \right) d(x,\partial_1V_n).\qedhere
\end{align*}
\end{proof}
The following bound will also be useful.
\begin{lemma}\label{lem:lemma2}
For any $z\in\partial_1V_n$
\begin{align*}
\E_{z}\left[\sum\limits_{k=0}^{\tau_1-1}k\one_{[S_k=y]} \right] \le C_2(m) d(y,\partial_1V_n) (m-1)^{-d(y,\partial_1V_n)},
\end{align*}
where
\begin{equation}\label{eq:C2}
C_2(m)=\left( \frac{2(m-1)^2}{(m-2)^3} + \frac{m(m-1)}{(m-2)^2} \right).
\end{equation}
\end{lemma}
\begin{proof}
We have 
\begin{align*}
\E_{z}\left[\sum\limits_{k=0}^{\tau_1-1}k\one_{[S_k=y]} \right] &  \le \E_{z}\left[\sum\limits_{k=0}^{\infty}k\one_{[S_k=y]} \right] \\
& = \frac{2(m-1) + d(z,y) m(m-2)}{(m-2)^3(m-1)^{d(z,y)-1}}\\
& \le \left( \frac{2(m-1)^2}{(m-2)^3} + \frac{m(m-1)}{(m-2)^2} \right) d(y,\partial_1V_n) (m-1)^{-d(y,\partial_1V_n)}.\qedhere
\end{align*}
\end{proof}

We now proceed to prove~Theorem~ \ref{thm:exp_max_0}.
\begin{proof}[Proof of Theorem~\ref{thm:exp_max_0}]
First we prove an upper bound for the expected maximum using a standard trick. Using Jensen's inequality and Lemma~\ref{lem:G_(x,x)_upper_bound} we have for any $\beta >0$
\begin{align*}
\E_n \left[ \max_{x\in V_n} \varphi_x \right]& \le \frac1{\beta} \log\left(\E_n\left[\sum_{x\in V_n}\e^{\beta \varphi_x}\right]\right) = \frac1{\beta} \log\left(\sum_{x\in V_n}\e^{\frac{\beta^2}2 G_n(x,x)}\right)\\
& \le \frac1{\beta} \log\left(\sum_{x\in V_n}\e^{\frac{\beta^2}2 G(o,o)}\right) = \frac{\beta}2 G(o,o) + \frac1{\beta} \log N.
\end{align*}
Optimizing over $\beta$ we obtain
\begin{align*}
\E_n \left[ \max_{x\in V_n} \varphi_x \right] \le \sqrt{2G(o,o) \log N}.
\end{align*}
Thus
\begin{align}\label{eq:limsup_exp_max_0}
\limsup_{n\to\infty} \frac{\E_n \left[ \max_{x\in V_n} \varphi_x \right]}{\sqrt{2\log N}} \le \sqrt{G(o,o)}.
\end{align}
Next we prove the lower bound for the limes inferior. We use a Gaussian comparison inequality on an appropriate set of vertices. For this we need a lower bound on $G_n(x,x)$ for $x$ in an appropriate subset of $V_n$. Using~\eqref{eq:G_n split}
we write
\begin{equation}\label{eq:G_n split_2}
G_n(x,y) = G(x,y) - \overline{E}_n(x,y) - E_n(x,y),
\end{equation}
where
\begin{align*}
\overline{E}_n(x,y):= \E_x\left[\sum\limits_{k=\tau_{0}}^\infty (k+1)\one_{[S_k=y]} \right].
\end{align*}
Our target is to obtain a bound for $\overline{E}_n(x,y)$. We have using Lemma~\ref{lemma:bound_tau}, Lemma~\ref{lem:lemma2} and \eqref{eq:lem_1.24} with $\z=1$
\begin{align}\label{eq:bar_E_n_bound}
\overline{E}_n(x,y) & = \E_x\left[\E_{S_{\tau_0}}\left[\sum\limits_{k=0}^\infty k\one_{[S_k=y]} \right]\right] + \E_x\left[(\tau_0+1)\E_{S_{\tau_0}}\left[\sum\limits_{k=0}^\infty\one_{[S_k=y]} \right]\right]\nonumber\\
& \le C_2(m) d(y,\partial_1V_n) (m-1)^{-d(y,\partial_1V_n)} + \frac{m-1}{m-2} (m-1)^{-d(y,\partial_1V_n)} \E_x[\tau_0+1]\nonumber\\
& \le C_2(m) d(y,\partial_1V_n) (m-1)^{-d(y,\partial_1V_n)} + \frac{m-1}{m-2} \left(1+C_1(m)d(x,\partial_1V_n)\right) (m-1)^{-d(y,\partial_1V_n)}.
\end{align}
To prove the bound for the limes inferior we define a subset $U_n$ of $V_n$ as follows. For each $z\in \partial_0 V_{n-2\lfloor\log n\rfloor}=\{x\in V_n: d(o,x)=n-2\lfloor\log n\rfloor \}$, choose exactly one $y_z \in \partial_0 V_{n-\lfloor\log n\rfloor}$ such that $d(z,\,y_z)=\lfloor\log n\rfloor$. Then define 
\begin{align*}
U_n:=\{y_z: z\in \partial_0 V_{n-2\lfloor\log n\rfloor}\}.
\end{align*}
Note that $U_n \subset \partial_0 V_{n-\lfloor\log n\rfloor}$ but $|U_n| = |\partial_0 V_{n-2\lfloor\log n\rfloor}| = m(m-1)^{n-2\lfloor\log n\rfloor-1}$.
Also from the definition of $U_n$ it follows that for any $x,y\in U_n$, $d(x,y) \ge 2\floor{\log n}$ and for any $x\in U_n$, $d(x,\partial_1 V_n)= \lfloor\log n\rfloor+1$. Now using the crude bound for $E_n(x,y)$ in \eqref{eq:E_n_bound} we have that for any $x,y\in U_n$
\begin{align*}
E_n(x,y) \le C_0(m)(\log n)^2 (m-1)^{-\log n} ,
\end{align*}
where $C_0(m)$ is a constant dependent on $m$ only.
Also from~\eqref{eq:bar_E_n_bound} we have for $x\in U_n$
\begin{align*}
\overline{E}_n(x,x)& \le C_2(m) \log n (m-1)^{-\log n} + \frac{m-1}{m-2} \left(1+C_1(m)\log n\right) (m-1)^{-\log n}\\
& \le  \frac{m-1}{m-2} \left(1+C_1(m)+C_2(m) \right)\log n (m-1)^{-\log n} .
\end{align*}
Using these bounds we get from~\eqref{eq:G_n split_2} that for any $y\in U_n$
\begin{align}\label{eq:G_n_lower bound}
G_n(y,y) &\ge G(o,o) - \frac{m-1}{m-2} \left(1+C_1(m)+C_2(m) \right)\log n (m-1)^{-\log n}\nonumber\\ 
&\qquad-  C_0(m)(\log n)^2 (m-1)^{-\log n}.
\end{align}
Also from Lemma~\ref{lem:G} we have for $y,\,y'\in U_n$
\begin{align*}
G(y,y') \le C_3(m) \log n (m-1)^{-2\log n}
\end{align*}
and hence
\begin{align}\label{eq:G_n_upper bound}
G_n(y,y') &\le G(y,y') + |E_n(y,y')|\nonumber\\
&\le C_3(m) \log n (m-1)^{-2\log n} + C_0(m) (\log n)^2 (m-1)^{-\log n}.
\end{align}
Now using~\eqref{eq:G_n_lower bound} and~\eqref{eq:G_n_upper bound} we have for $y,y'\in U_n$
\begin{align}\label{eq:for sud-fer}
\E_n\left[(\vr_y-\vr_{y'})^2\right]& = G_n(y,y) + G_n(y',y') - 2 G_n(y,y')\nonumber\\
&\ge 2\left[ G(o,o) - \frac{m-1}{m-2} \left(1+C_1(m)+C_2(m) \right)\log n (m-1)^{-\log n}  \right.\nonumber\\
&\left.- C_3(m) \log n (m-1)^{-2\log n} -C_0(m) (\log n)^2 (m-1)^{-\log n}\right].
\end{align}
We define
\begin{align*}
\gamma(n,m)&:= \left[ G(o,o) - \frac{m-1}{m-2} \left(1+C_1(m)+C_2(m) \right)\log n (m-1)^{-\log n}  \right.\\
&\left.- C_3(m) \log n (m-1)^{-2\log n} - C_0(m) (\log n)^2 (m-1)^{-\log n}\right].
\end{align*}
Note that $\gamma(n,m)\to G(o,o)$ as $n\to\infty$. Suppose $n$ is large enough so that $\gamma(n,m) >0$. Let $(\xi_x)_{x\in U_n}$ be~i.i.d.~centered Gaussian random variables with variance $\gamma(n,m)$. Then from~\eqref{eq:for sud-fer} we have
\begin{align*}
\E_n\left[(\vr_x-\vr_y)^2\right] \ge \E\left[(\xi_x-\xi_y)^2\right].
\end{align*}
Therefore by Sudakov-Fernique inequality~\cite[Theorem 2.~2.~3]{AdlerTaylor} we have
\begin{align*}
\E_n \left[ \max_{x\in U_n} \varphi_x \right] \ge \E \left[ \max_{x\in U_n} \xi_x \right].
\end{align*}
As $U_n\subset V_n$, we get
\begin{align*}
\liminf_{n\to\infty}\frac{\E_n \left[ \max_{x\in V_n} \varphi_x \right]}{\sqrt{2\log N}} \ge \liminf_{n\to\infty}\frac{\E_n \left[ \max_{x\in U_n} \xi_x \right]}{\sqrt{2\log |U_n|}} \sqrt{\frac{\log |U_n|}{\log N}}.
\end{align*}
But 
\begin{align*}
\frac{\log |U_n|}{\log N} = \frac{n-2\log n}{n} \overset{n\to\infty}{\to} 1.
\end{align*}
{
Therefore
\begin{align}\label{eq:liminf_exp_max_0}
\liminf_{n\to\infty}\frac{\E_n \left[ \max_{x\in V_n} \varphi_x \right]}{\sqrt{2\log N}} \ge \sqrt{G(o,o)}.
\end{align}
So the result follows now combining the lower bound \eqref{eq:liminf_exp_max_0} with the upper bound~\eqref{eq:limsup_exp_max_0}.}

\end{proof}

\section{Proof of Theorem~\ref{thm:max_0}}\label{section:thm:max_0}
In this section we prove Theorem~\ref{thm:max_0}. Before proving, we state two estimates which are crucially used in the proof. Recall the crucial error term in the Vanderbei's representation \eqref{error:rep}

\begin{align*}
E_n(x,y)=\lim\limits_{t\to\infty}\E_x\left[\sum\limits_{j=1}^{\eta_t}(-1)^{j-1}M_{j-1}\sum\limits_{k=\tau_{j-1}}^{\tau_j-1}(k-\tau_{j-1})\one_{[S_k=y]} \right].
\end{align*}

Previously in \eqref{eq:E_n_bound} we showed that 
$$E_n(x,y) \le C d(x,\partial_1V_n)d(y, \partial_1V_n)\min\{ (m-1)^{-\frac{1}{2}d(x, \partial_1V_n)}, (m-1)^{-\frac{1}{2}d(y, \partial_1V_n)}\}.$$
 This bound does not say anything about the dependency of the error on $d(x,y)$. We improve the bound to get a dependence on the distance between the two points and this is crucial for our proof. 

\begin{lemma}\label{lem:E_n_finer}
Let $x,y\in V_n$. For $m \ge 5$ and any $0\le J_0<\infty$ we have
\begin{align}\label{eq:E_n_finer} 
&|E_n(x,y)| \le J_0(2J_0+1)^{4J_0+2} 3^{4J_0^2} \left(\frac{4(m-1)}{m-2}\right)^{2J_0+1} \frac{m-1}{m-2}d^{2J_0+1} (m-1)^{-d(x,y)}\nonumber \\
&+ \frac{C_1(m) C_2(m)}{ \left(1-\frac{C_1(m)}{m}\right)}d(x,\partial_1V_n)d(y,\partial_1V_n) (m-1)^{-\max\{d(x,\partial_1V_n), d(y,\partial_1V_n)\}} \left(\frac{C_1(m)}{m}\right)^{J_0},
\end{align}
where $C_1(m)$ and $C_2(m)$ are constants defined in \eqref{eq:C1} and \eqref{eq:C2} respectively. 
\end{lemma}

As the proof requires some lengthy computations we devote Section~\ref{section:lem:E_n_finer} to it. {Note that when we take $J_0=0$, the above bound improves \eqref{eq:E_n_bound} and we understand the constants better as well.}
In the proof of Theorem~\ref{thm:max_0} we shall use $J_0=0$ and $J_0=\log(d(x,y))$.

We know that $G_n(x,x)\to G(o,o)$ but we do not know if this convergence is uniform as the error term depends on the distance of $x$ from the boundary. However we see that for $m\ge 10$ we can bound $G_n(x,x)$ uniformly from below. 
\begin{lemma}
For $m\ge 10$ there is a positive constant $C_3(m)$ which converges to $1$ as $m\to\infty$  such that
\begin{align}\label{eq:var_lower_bound_0}
\inf_{x\in V_n} G_n(x,x) \ge C_3(m).
\end{align}

\end{lemma}
\begin{proof}
Note that $\overline{G}_n(x,x)\ge 1$.
Taking $J_0=0$ in~\eqref{eq:E_n_finer} we get for $m\ge 5$
\begin{align*}
|E_n(x,x)| \le \frac{C_1(m) C_2(m)}{ \left(1-\frac{C_1(m)}{m}\right)}d(x,\partial_1V_n)^2 (m-1)^{-d(x,\partial_1V_n)}
 \le \frac{C_1(m) C_2(m)}{ (m-1)\left(1-\frac{C_1(m)}{m}\right)}.
\end{align*}
Therefore
\begin{align*}
G_n(x,x) \ge 1- \frac{C_1(m) C_2(m)}{ (m-1)\left(1-\frac{C_1(m)}{m}\right)} =:C_3(m).
\end{align*}
But by definition of $C_1(m)$ and $C_2(m)$ in \eqref{eq:C1} and \eqref{eq:C2} respectively, it follows 
\begin{align*}
\lim_{m\to\infty}\frac{C_1(m) C_2(m)}{ (m-1)\left(1-\frac{C_1(m)}{m}\right)} {=} 0.
\end{align*}
One can observe that the function $m \mapsto C_3(m)$ is an increasing function for $m\ge 5$ and for $m=9$ and $10$ we compute that $C_3(9)=-0.06 $ and $C_3(10)= 0.2$. Hence for $m\ge 10$ we have $G_n(x,x) \ge C_3(m)>0$.
\end{proof}

We now prove Theorem~\ref{thm:max_0}. In the proof we again use the comparison theorem of \cite{hj90}.
\begin{proof}[Proof of Theorem~\ref{thm:max_0}]
We set
\begin{align*}
 u_n(\theta):= A_n\theta + B_n,\,\,\theta\in\R.
\end{align*}
From the proof of Theorem~\ref{thm:max} we observe that in this case it suffices to prove 
\begin{align}\label{eq:main_0}
\sum_{x,y\in V_n, x\neq y}|\mathbf{Cov}\left(\one_{[\psi_x > u_n(\theta)]}, \one_{[\psi_y > u_n(\theta)]}\right)| \to 0
\end{align}
for all $\theta$. 

We define $R_n(x,y):= \E_n[\psi_x \psi_y]$. First we obtain a bound for $|R_n(x,y)|$. By using Lemma~\ref{lem:G}, Lemma~\ref{lem:E_n_finer} and~\eqref{eq:var_lower_bound_0} we obtain
\begin{align}\label{eq:R_n bound}
|R_n(x,y)| &= \frac{|G_n(x,y)|}{\sqrt{G_n(x,x)G_n(y,y)}}\nonumber\\
& \le \frac{1}{C_3(m)}\left[ G(x,y) + |E_n(x,y)| \right]\nonumber\\
& \le \frac{1}{C_3(m)} \left[  \frac{(d(x,y)+1) m(m-1)(m-2) - 2(m-1)}{(m-2)^3(m-1)^{d(x,y)}} \nonumber \right.\\
&\left.+ J_0(2J_0+1)^{4J_0+2} 3^{4J_0^2} \left(\frac{4(m-1)}{m-2}\right)^{2J_0+1} \frac{m-1}{m-2}d(x,y)^{2J_0+1} (m-1)^{-d(x,y)}\nonumber \right.\\
&+ \left.\frac{C_1(m) C_2(m)}{ \left(1-\frac{C_1(m)}{m}\right)}d(x,\partial_1V_n)d(y,\partial_1V_n) (m-1)^{-\max\{d(x,\partial_1V_n), d(y,\partial_1V_n)\}} \left(\frac{C_1(m)}{m}\right)^{J_0} \right],
\end{align}
where $0\le J_0 <\infty$. Taking $J_0=0$ in~\eqref{eq:R_n bound} we observe that for all distinct $x,y\in V_n$ 
\begin{align}\label{eq:above}
|R_n(x,y)| \le \frac{1}{C_3(m)}  \left[  \frac{2m(m-1)(m-2) - 2(m-1)}{(m-2)^3(m-1)} + \frac{C_1(m) C_2(m)}{ (m-1)\left(1-\frac{C_1(m)}{m}\right)} \right].
\end{align}
It is easy to check that the function $$m\mapsto \frac{1}{C_3(m)}  \left[  \frac{2m(m-1)(m-2) - 2(m-1)}{(m-2)^3(m-1)} + \frac{C_1(m) C_2(m)}{ (m-1)\left(1-\frac{C_1(m)}{m}\right)} \right]$$
is decreasing for $m\ge 10$. For $m=13$ and $14$ we evaluate the above expression as $1.13$ and $0.89$ respectively. Therefore we conclude that for all distinct $x,\,y\in V_n$ and for $m\ge 14$
\begin{align}\label{eq:R_n bound2}
|R_n(x,y)| & \le \eta 
\end{align}
for some fixed $\eta$ with $0<\eta<1$. We are now ready to prove~\eqref{eq:main_0}. Let $\theta\in\R$ be fixed. We will use the following bounds which are obtained from Lemma 3.4 of~\cite{hj90}.
\begin{lemma}
For $x,y\in V_n$ the following hold.
\begin{enumerate}
\item If $0 \le R_n(x,y) <1$,
\begin{align}0\le \mathbf{Cov}\left(\one_{[\psi_x > u_n(\theta)]}, \one_{[\psi_y > u_n(\theta)]}\right) \le C N^{-\frac{2}{1+R_n(x,y)}} \label{eq:bound1_0}.\end{align}
\item If $0 \le R_n(x,y) \le 1$,
\begin{align}0\le \mathbf{Cov}\left(\one_{[\psi_x > u_n(\theta)]}, \one_{[\psi_y > u_n(\theta)]}\right) \le C R_n(x,y) N^{-2}\log N \e^{2R_n(x,y)\log N}.\label{eq:bound2_0}\end{align}
\item If $-1 \le R_n(x,y) <0$,
\begin{align}0\ge \mathbf{Cov}\left(\one_{[\psi_x > u_n(\theta)]}, \one_{[\psi_y > u_n(\theta)]}\right) \ge -C N^{-2}\label{eq:bound3_0}.\end{align}
\item If $-1\le R_n(x,y) \le 0$,
\begin{align}0\ge \mathbf{Cov}\left(\one_{[\psi_x > u_n(\theta)]}, \one_{[\psi_y > u_n(\theta)]}\right) \ge -C |R_n(x,y)| N^{-2}\log N.\label{eq:bound4_0}\end{align}
 \end{enumerate}
\end{lemma}
We write
\begin{align*}
\sum_{x,y\in V_n, x\neq y}&|\mathbf{Cov}\left(\one_{[\psi_x > u_n(\theta)]}, \one_{[\psi_y > u_n(\theta)]}\right)| \\
&=\sum_{x,y\in V_n, x\neq y}|\mathbf{Cov}\left(\one_{[\psi_x > u_n(\theta)]}, \one_{[\psi_y > u_n(\theta)]}\right)| \one_{[0\le R_n(x,y) \le 1]}\\
&+ \sum_{x,y\in V_n, x\neq y}|\mathbf{Cov}\left(\one_{[\psi_x > u_n(\theta)]}, \one_{[\psi_y > u_n(\theta)]}\right)| \one_{[-1\le R_n(x,y) < 0]}=: T_1 + T_2.
\end{align*}
We show that both $T_1$ and $T_2$ go to zero as $n$ tends to infinity. First we consider $T_1$. Let us choose $\eps$ such that $0 <\eps <(1-\eta)/(1+\eta) <1$, where $\eta$ is the same as in~\eqref{eq:R_n bound2}. We now split $T_1$ as
\begin{align}\label{eq:T_1_0}
T_1&=\sum_{k=1}^{2n} \sum_{x,y\in V_n, d(x,y)=k}|\mathbf{Cov}\left(\one_{[\psi_x > u_n(\theta)]}, \one_{[\psi_y > u_n(\theta)]}\right)|\one_{[0\le R_n(x,y) \le 1]}\nonumber\\
& = \sum_{k=1}^{\floor{n\eps}} \sum_{x,y\in V_n, d(x,y)=k}|\mathbf{Cov}\left(\one_{[\psi_x > u_n(\theta)]}, \one_{[\psi_y > u_n(\theta)]}\right)|\one_{[0\le R_n(x,y) \le 1]} \nonumber\\
&\qquad+ \sum_{k=\floor{n\eps}+1}^{2n} \sum_{x,y\in V_n, d(x,y)=k}|\mathbf{Cov}\left(\one_{[\psi_x > u_n(\theta)]}, \one_{[\psi_y > u_n(\theta)]}\right)|\one_{[0\le R_n(x,y) \le 1]}.
\end{align}
Using bound~\eqref{eq:bound1_0} with~\eqref{eq:R_n bound2} and Lemma~\ref{lem:C_k} we observe that the first term of~\eqref{eq:T_1_0} is bounded by 
\begin{align*}
C \sum_{k=1}^{\floor{n\eps}}  (m-1)^{n+\floor{\frac{k}2}} N^{-\frac2{1+\eta}}
\le C (m-1)^{n+\floor{\frac{n\eps}2} -\frac{2n}{1+\eta}}
\end{align*} 
which goes to zero as $n$ tends to infinity by the choice of $\eps$.

For the second term in~\eqref{eq:T_1_0} we use the bound~\eqref{eq:bound2_0} together with the bound~\eqref{eq:R_n bound} with $J_0=\log(d(x,y))$. We get that the second term is bounded by
\begin{align*}
&\sum_{k=\floor{n\eps}+1}^{2n} \sum_{x,y\in V_n, d(x,y)=k} C R_n(x,y) N^{-2}\log N \e^{2R_n(x,y)\log N} \\
& \le C \sum_{k=\floor{n\eps}+1}^{2n} (m-1)^{n+\floor{\frac{k}2}} N^{-2}\log N \left(C\frac{k}{(m-1)^k} + B_k\right) \exp\left[2\log N\left(C\frac{k}{(m-1)^k} + B_k\right)\right],
\end{align*}
where 
\begin{align*}
B_k:= \log k (2\log k+1)^{4\log k+2} 3^{4(\log k)^2} \left(\frac{4(m-1)}{m-2}\right)^{2\log k+1} \frac{m-1}{m-2}k^{2\log k+1} (m-1)^{-k} + C \left(\frac{C_1(m)}{m}\right)^{\log k} .
\end{align*}
We now use the following fact for $B_k$ whose proof is given in the end of this section.
\begin{claim}\label{claim:B_k}
For large $n$ and for all $k\ge \floor{n\eps}$
\begin{align*}
B_k \le B_{\floor{n\eps}} .
\end{align*}
\end{claim} 
Using the above claim we get that the second term of~\eqref{eq:T_1_0} is bounded by
 \begin{align*}
 C n \left(\frac{\floor{n\eps}}{(m-1)^{\floor{n\eps}}} + B_{\floor{n\eps}}\right) \exp\left[Cn\left(\frac{\floor{n\eps}}{(m-1)^{\floor{n\eps}}} + B_{\floor{n\eps}}\right)\right].
\end{align*}
Note that to show that the above bound goes to zero as $n$ tends to infinity it is enough to prove $nB_{\floor{n\eps}} \to 0$ as $n\to\infty$. We have 
\begin{align*}
nB_{\floor{n\eps}} &= \exp\left[ \log n +\log\log \floor{n\eps}+ (4\log \floor{n\eps}+2)\log(2\log \floor{n\eps}+1) + (4(\log \floor{n\eps})^2) \log 3\right.\\
& \left.+ (2\log \floor{n\eps}+1)\log\left(\frac{4(m-1)}{m-2}\right) + \log(\frac{m-1}{m-2}) + (2\log \floor{n\eps}+1)\log \floor{n\eps} -\floor{n\eps}\log(m-1)\right] \\
&+ \exp\left[ \log n + \log C + \log \floor{n\eps} \log\left(\frac{C_1(m)}{m}\right)\right]\overset{n\to\infty}{\to}0.
\end{align*}

Here the magnitude of $m$ is used to get that \begin{equation}\exp\left[ \log n + \log C + \log \floor{n\eps} \log\left(\frac{C_1(m)}{m}\right)\right]\overset{n\to\infty}{\to}0.\label{eq:conv_long}\end{equation}
We observe that the function $m \mapsto \log(C_1(m)/m)$ is a decreasing function and moreover that $\log(C_1(9)/9)=-1.08$. Hence~\eqref{eq:conv_long} holds for all $m\ge 9$.
Thus we proved that $T_1\to 0$ as $n\to\infty$. 

Next we consider $T_2$. For $T_2$ we use the bounds~\eqref{eq:bound3_0},~\eqref{eq:bound4_0} together with the bound~\eqref{eq:R_n bound} with $J_0=\log(d(x,y))$ to get
\begin{align*}
T_2&=\sum_{k=1}^{2n} \sum_{x,y\in V_n, d(x,y)=k}|\mathbf{Cov}\left(\one_{[\psi_x > u_n(\theta)]}, \one_{[\psi_y > u_n(\theta)]}\right)|\one_{[-1\le R_n(x,y) < 0]}\\
& = \sum_{k=1}^{n-1} \sum_{x,y\in V_n, d(x,y)=k}|\mathbf{Cov}\left(\one_{[\psi_x > u_n(\theta)]}, \one_{[\psi_y > u_n(\theta)]}\right)|\one_{[-1\le R_n(x,y) < 0]} \nonumber\\
&\qquad+ \sum_{k=n}^{2n} \sum_{x,y\in V_n, d(x,y)=k}|\mathbf{Cov}\left(\one_{[\psi_x > u_n(\theta)]}, \one_{[\psi_y > u_n(\theta)]}\right)|\one_{[-1\le R_n(x,y) < 0]}\\
& \le C \sum_{k=1}^{n-1} (m-1)^{n+\floor{\frac{k}2}-2n} + C \sum_{k=n}^{2n} (m-1)^{n+\floor{\frac{k}2}-2n} n \left(C\frac{k}{(m-1)^k} + B_k\right) .
\end{align*}
Clearly, the first part in the above bound goes to zero as $n$ tends to infinity. For the second part we use the fact that $B_k \le B_n$ for all $k\ge n$ for large $n$ which can be proved similarly as the Claim~\ref{claim:B_k}. Then we get that the second part is bounded by $C n(\frac{n}{(m-1)^n} + B_n)$ which can be shown to go to zero as $n\to\infty$ similarly as in the case of $T_1$. Thus $T_2\to 0$ as $n\to\infty$. This completes the proof of~\eqref{eq:main_0}.
\end{proof}

We now prove Claim~\ref{claim:B_k}. 
\begin{proof}[Proof of Claim~\ref{claim:B_k}]
We define for $t\ge 0$
\begin{align*}
F(t)&:= \log t (2\log t+1)^{4\log t+2} 3^{4(\log t)^2} \left(\frac{4(m-1)}{m-2}\right)^{2\log t+1} \frac{m-1}{m-2} t^{2\log t+1} (m-1)^{-t} \\
&\qquad\qquad+ C \left(\frac{C_1(m)}{m}\right)^{\log t} \\
&= \exp\left[ \log\log t+ (4\log t+2)\log(2\log t+1) + (4(\log t)^2) \log 3 \right.\\
&\left.\qquad+ (2\log t+1)\log\left(\frac{4(m-1)}{m-2}\right) + \log(\frac{m-1}{m-2}) + (2\log t+1)\log t -t\log(m-1)\right] \\
&\qquad+ \exp\left[ \log C + \log t \log\left(\frac{C_1(m)}{m}\right)\right].
\end{align*}
Then
\begin{align*}
F{'}(t) &= \exp\left[ \log\log t+ (4\log t+2)\log(2\log t+1) + (4(\log t)^2) \log 3 \right.\\
&\left.\qquad+ (2\log t+1)\log\left(\frac{4(m-1)}{m-2}\right) + \log(\frac{m-1}{m-2}) + (2\log t+1)\log t -t\log(m-1)\right] \\
&\qquad\left[ \frac1{t\log t} + \frac4{t} \log(2\log t+1) + \frac{2(4\log t+2)}{t(2\log t+1)} + \frac{8\log 3 \log t}{t} + \frac{2}{t} \log\left(\frac{4(m-1)}{m-2}\right)\right.\\ 
&\left.\qquad+ \frac{2\log t}{t} + \frac{2\log t +1}{t} -\log(m-1) \right]\\
& \qquad+ \exp\left[ \log C + \log t \log\left(\frac{C_1(m)}{m}\right)\right] \left[ \frac1{t} \log\left(\frac{C_1(m)}{m}\right)\right].
\end{align*}
This shows that $F{'}(t) <0$ for all $t \ge t_0$ for some $t_0$. Hence $F$ is decreasing on $[t_0, \infty)$.
Therefore $F(k)= B_k\le F(\floor{n\eps})=B_{\floor{n\eps}}$ for all $k\ge \floor{n\eps}$ for large $n$.
\end{proof}

\section{Finer estimate on the error and proof of Lemma~\ref{lem:E_n_finer}}\label{section:lem:E_n_finer}
In this section we will derive a finer estimate on the error $E_n(x,y)$ which was crucially used to prove Theorem~\ref{thm:max_0}. We look at each individual term of the series which appears in $E_n(x,y)$ and find a better bound than what we have before. 
\begin{proof}[Proof of Lemma~\ref{lem:E_n_finer}]
Recall from~\eqref{error:rep}
\begin{align*}
E_n(x,y)=\lim\limits_{t\to\infty}\E_x\left[\sum\limits_{j=1}^{\eta_t}(-1)^{j-1}M_{j-1}\sum\limits_{k=\tau_{j-1}}^{\tau_j-1}(k-\tau_{j-1})\one_{[S_k=y]} \right].
\end{align*}
Now conditioning on $\eta_t$ we get
\begin{align}\label{eq:E_n_poisson}
|E_n(x,y)|&=\left | \lim\limits_{t\to\infty}\sum_{\ell=0}^\infty \e^{-t}\frac{t^\ell}{\ell!}\sum\limits_{j=1}^{\ell}(-1)^{j-1}\E_x\left[M_{j-1}\sum\limits_{k=\tau_{j-1}}^{\tau_j-1}(k-\tau_{j-1})\one_{[S_k=y]} \right]\right |\nonumber\\
&  \le \lim\limits_{t\to\infty}\sum_{\ell=0}^\infty \e^{-t}\frac{t^\ell}{\ell!}\sum\limits_{j=1}^{\ell} a_j,
\end{align}
where
\begin{align*}
a_j:= \E_x\left[M_{j-1}\sum\limits_{k=\tau_{j-1}}^{\tau_j-1}(k-\tau_{j-1})\one_{[S_k=y]} \right],\,\,\,j\ge 1.
\end{align*}
We now obtain bounds for $a_j$. First we bound each of the $a_j$'s in terms of the distance of $x$ and $y$ from the boundary of $V_n$ and then we bound the $a_j$'s in terms of the distance $d(x,y)$.

\begin{claim}\label{claim:estimateaj}
For all $j\ge 1$ the following two estimates hold.
\begin{itemize}[leftmargin=2em]
\item[(a)] {\bf Bound in terms of distance from boundary.} 
\begin{align}\label{eq:a_j_1}
a_j \le C_1(m) C_2(m) d(x,\partial_1V_n)d(y,\partial_1V_n) (m-1)^{-d(y,\partial_1V_n)} \left(\frac{C_1(m)}{m}\right)^{j-1}.
\end{align}
\item[(b)] {\bf Bound in terms of $d(x,y)$.}
\begin{align}\label{eq:a_j_2}
a_j&\le (2j+1)^{4j+2} 3^{4j^2} \left(\frac{4(m-1)}{m-2}\right)^{2j+1} \frac{m-1}{m-2}(d(x,y))^{2j+1} (m-1)^{-d(x,y)}.
\end{align}
\end{itemize}
\end{claim}

To obtain a finer estimate on $|E_n(x,y)|$ let us fix $J_0\in [0,\infty)$ with the notion that when $J_0=0$, the sum $\sum_{j=1}^{J_0}$ is $0$. Using the bounds~\eqref{eq:a_j_1} and~\eqref{eq:a_j_2} we have
\begin{align*}
\sum\limits_{j=1}^{\ell} a_j &\le \sum\limits_{j=1}^{\infty} a_j = \sum\limits_{j=1}^{J_0} a_j + \sum\limits_{j=J_0+1}^{\infty} a_j\\
& \le \sum\limits_{j=1}^{J_0}  (2j+1)^{4j+2} 3^{4j^2} \left(\frac{4(m-1)}{m-2}\right)^{2j+1} \frac{m-1}{m-2}(d(x,y))^{2j+1} (m-1)^{-d(x,y)} \\
&\qquad+\sum\limits_{j=J_0+1}^{\infty} C_1(m) C_2(m) d(x,\partial_1V_n)d(y,\partial_1V_n) (m-1)^{-d(y,\partial_1V_n)} \left(\frac{C_1(m)}{m}\right)^{j-1}\\
& \le J_0(2J_0+1)^{4J_0+2} 3^{4J_0^2} \left(\frac{4(m-1)}{m-2}\right)^{2J_0+1} \frac{m-1}{m-2}(d(x,y))^{2J_0+1} (m-1)^{-d(x,y)}\\
&+ C_1(m) C_2(m) \left(1-\frac{C_1(m)}{m}\right)^{-1}d(x,\partial_1V_n)d(y,\partial_1V_n) (m-1)^{-d(y,\partial_1V_n)} \left(\frac{C_1(m)}{m}\right)^{J_0}.
\end{align*}
Here we have used the fact that $C_1(m)/m <1$ for $m \ge 5$. Indeed, one can observe that the function $m \mapsto C_1(m)/m$ is a decreasing function and for $m=4$ and $5$ we compute that $C_1(4)/4=1.39$ and $C_1(5)/5 = 0.87$.

Now from~\eqref{eq:E_n_poisson} it follows that
\begin{align*}
|E_n(x,y)| &\le J_0(2J_0+1)^{4J_0+2} 3^{4J_0^2} \left(\frac{4(m-1)}{m-2}\right)^{2J_0+1} \frac{m-1}{m-2}(d(x,y))^{2J_0+1} (m-1)^{-d(x,y)}\\
& + C_1(m) C_2(m) \left(1-\frac{C_1(m)}{m}\right)^{-1}d(x,\partial_1V_n)d(y,\partial_1V_n) (m-1)^{-d(y,\partial_1V_n)} \left(\frac{C_1(m)}{m}\right)^{J_0}.
\end{align*}
From symmetry we conclude that~\eqref{eq:E_n_finer} holds. We are now left to prove Claim~\ref{claim:estimateaj}.
 \paragraph{\bf Proof of Claim~\ref{claim:estimateaj}}
First we prove part (a). The proof involves the successive use of the strong Markov property. {We have
\begin{align*}
a_j&= \E_x\left[M_{j-1}\sum\limits_{k=\tau_{j-1}}^{\tau_j-1}(k-\tau_{j-1})\one_{[S_k=y]} \right]\\
&= \E_x\left[ \tau_0 \E_{S_{\tau_0}}\left[\prod_{i=1}^{j-1} (\tau_i -\tau_{i-1}-1) \sum\limits_{k=\tau_{j-1}}^{\tau_j-1}(k-\tau_{j-1})\one_{[S_k=y]} \right] \right] \\
&= \E_x\left[ \tau_0 \E_{S_{\tau_0}}\left[(\tau_1-1) \E_{S_{\tau_1}}\left[\prod_{i=1}^{j-2} (\tau_i -\tau_{i-1}-1) \sum\limits_{k=\tau_{j-2}}^{\tau_{j-1}-1}(k-\tau_{j-2})\one_{[S_k=y]} \right] \right]\right] .
\end{align*}
Iteratively using the strong Markov property we obtain
\begin{align}\label{eq:manybrackets}
a_j= \E_x\left[ \tau_0\underbrace{ \E_{S_{\tau_0}}\left[(\tau_1-1) \E_{S_{\tau_1}} \left[(\tau_1-1) \E_{S_{\tau_1}}\left[(\tau_1-1) \ldots \E_{S_{\tau_1}}\left[(\tau_1-1)\E_{S_{\tau_1}}\left[\sum\limits_{k=0}^{\tau_{1}-1}k\one_{[S_k=y]} \right] \right]\right]\right]\right] }_{j- \text{ many expectations}}\right] .
\end{align}
Note that for any $z\in\partial_1 V_n$
\begin{align*}
\E_{z}\left[\tau_1-1\right]&=  \left[\E_z[\tau_1-1 | S_1\in V_n] \prob_z(S_1\in V_n) + \E_z[\tau_1-1 | S_1\in V_n^c] \prob_z(S_1\in V_n^c) \right]\nonumber\\
& = \frac1{m} \E_z[\tau_1-1 | S_1\in V_n]\overset{\eqref{eq:bound_tau}}{\le} \frac{C_1(m)}{m}. \label{eq:bound_m_tau}
\end{align*}
This together with Lemma~\ref{lem:lemma2} and Lemma~\ref{lemma:bound_tau} gives the bound~\eqref{eq:a_j_1}.
}

\paragraph{Part (b)}

We obtain a bound for $a_j$ in terms of the distance between $x$ and $y$. Let $p_k(z, w)= \prob_z[ S_k=w]$ be the $k$-step transition probability. We show it in two steps. First we show
\begin{equation}\label{eq:bound_a_j}
a_j \le \sum\limits_{k=0}^\infty k^{2j+1} p_{k}(x,y)
\end{equation}
and then we express $\sum\limits_{k=0}^\infty k^{2j+1} p_{k}(x,y)$ in terms of the derivatives of $g(\z)=\Gamma(x,y|\z)=\sum_{k=0}^\infty \prob_x\left[ S_k=y\right]\z^k$. We explicitly compute these derivatives in Subsection \ref{subsec:derivative}.  

{
We have from~\eqref{eq:manybrackets} that
\begin{align*}
a_j&= \E_x\left[ \tau_0\underbrace{ \E_{S_{\tau_0}}\left[(\tau_1-1) \E_{S_{\tau_1}} \left[(\tau_1-1) \E_{S_{\tau_1}}\left[(\tau_1-1) \ldots \E_{S_{\tau_1}}\left[(\tau_1-1)\E_{S_{\tau_1}}\left[\sum\limits_{k=0}^{\tau_{1}-1}k\one_{[S_k=y]} \right] \right]\right]\right]\right] }_{j- \text{ many expectations}}\right] \\
& = \sum\limits_{z_1,\,\ldots,\,z_j\in\partial_1 V_n}  \E_x\left[\tau_0\one_{[S_{\tau_0}=z_1]}\E_{z_1}\left[(\tau_1-1)\one_{[S_{\tau_1}=z_2]} \ldots \E_{z_{j-1}}\left[(\tau_1-1)\one_{[S_{\tau_1}=z_j]}\E_{z_j}\left[\sum\limits_{k=0}^{\tau_1-1}k\one_{[S_k=y]} \right]\right]\right]\right]\\
& \le \sum\limits_{z_1,\,\ldots,\,z_j\in\partial_1 V_n}  \sum\limits_{\ell_0=0}^\infty \ell_0 p_{\ell_0}(z_j,y)\E_x\left[\tau_0\one_{[S_{\tau_0}=z_1]}\E_{z_1}\left[(\tau_1-1)\one_{[S_{\tau_1}=z_2]} \ldots \E_{z_{j-1}}\left[(\tau_1-1)\one_{[S_{\tau_1}=z_j]}\right]\right]\right]\\
& \le \sum\limits_{z_1,\,\ldots,\,z_j\in\partial_1 V_n}  \sum\limits_{\ell_0=0}^\infty \ell_0 p_{\ell_0}(z_j,y) \sum\limits_{\ell_1=0}^\infty \ell_1 p_{\ell_1}(z_j, z_{j-1}) \ldots \sum\limits_{\ell_j=0}^\infty \ell_j p_{\ell_j}(x,z_1)\\
& \le \sum\limits_{\ell_0,\,\ell_1,\,\ldots\,\ell_j=0}^\infty  \ell_0\ell_1\ldots\ell_j p_{\ell_0+\ell_1+ \ldots+\ell_j}(x,y)\\
&= \sum\limits_{k=0}^\infty \sum\limits_{\ell_1=0}^k \sum\limits_{\ell_2=0}^{k-\ell_1} \ldots \sum\limits_{\ell_j=0}^{k-(\ell_1+\ldots+\ell_{j-1})} (k-(\ell_1+\ldots+\ell_j))\ell_1\ldots \ell_j p_{k}(x,y) \le \sum\limits_{k=0}^\infty k^{2j+1} p_{k}(x,y).
\end{align*}

}

We now use the bound on the derivatives of $g$ from~\eqref{eq:g_deriv_bound} in~\eqref{eq:bound_a_j} to obtain a bound for $a_j$ in terms of $d(x,y)$.  For that we first write $k^{\ell}$ in terms of $\prod_{i=0}^{i_0}(k-i)$, $i_0=0,\,1,\,\ldots,\, \ell-1$. {We observe that}
\begin{align*}
&k^2=k(k-1) + k,\\
&k^3= k\left(k(k-1) + k\right) = \prod_{i=0}^{2}(k-i) + (2+1) \prod_{i=0}^{1}(k-i) +k,\\
&k^4=k \left( \prod_{i=0}^{2}(k-i) + 3\prod_{i=0}^{1}(k-i) +k \right) \\
&\quad= \prod_{i=0}^{3}(k-i) + (3+3) \prod_{i=0}^{2}(k-i) + (2\times 3+1) \prod_{i=0}^{1}(k-i)+k.\\
\end{align*}
In general we have that for any $k,\ell\ge 1$ 
\begin{align*}
k^{\ell}= \alpha^{(\ell)}_{\ell-1} \prod_{i=0}^{\ell-1}(k-i) + \alpha^{(\ell)}_{\ell-2} \prod_{i=0}^{\ell-2}(k-i) +\cdots +\alpha^{(\ell)}_{1} \prod_{i=0}^{1}(k-i) +\alpha^{(\ell)}_{0} k,
\end{align*} 
where the coefficients $\alpha^{(\ell)}_{i}$ for all $\ell \ge 1$ and $i=0,1,\ldots, \ell-1$ are given recursively as follows
\begin{align*}
&\alpha^{(\ell)}_{0}=\alpha^{(\ell)}_{\ell-1}=1,\\
&\alpha^{(\ell)}_{i} = (i+1)\alpha^{(\ell-1)}_{i} + \alpha^{(\ell-1)}_{i-1},\,\,\,1\le i \le \ell-2.
\end{align*}
It follows that for all $\ell \ge 1$ and $i=0,\,1,\,\ldots, \ell-1$
\begin{align}\label{eq:xi_bound}
\alpha^{(\ell)}_{i} \le \ell! \le \ell^\ell.
\end{align}
Now from~\eqref{eq:bound_a_j} we have
\begin{align*}
a_j &\le \sum\limits_{k=0}^\infty k^{2j+1} p_{k}(x,y)\\
&=\sum\limits_{k=0}^\infty p_{k}(x,y)\left[\alpha^{(2j+1)}_{2j} \prod_{i=0}^{2j}(k-i) + \alpha^{(2j+1)}_{2j-1} \prod_{i=0}^{2j-1}(k-i) +\cdots +\alpha^{(2j+1)}_{1} \prod_{i=0}^{1}(k-i) +\alpha^{(2j+1)}_{0} k \right]\\
&= \alpha^{(2j+1)}_{2j} g^{(2j+1)}(1) + \alpha^{(2j+1)}_{2j-1} g^{(2j)}(1)+\cdots +\alpha^{(2j+1)}_{1} g^{(2)}(1)+\alpha^{(2j+1)}_{0} g^{(1)}(1) .
\end{align*}
Now using~\eqref{eq:xi_bound} and~\eqref{eq:g_deriv_bound} we obtain
\begin{align*}
a_j &\le (2j+1) (2j+1)^{2j+1} 3^{(2j)^2} (2j)^{2j}\left(\frac{4(m-1)}{m-2}\right)^{2j+1} \frac{m-1}{m-2}(d(x,y))^{2j+1} (m-1)^{-d(x,y)}\nonumber\\
& \le (2j+1)^{4j+2} 3^{4j^2} \left(\frac{4(m-1)}{m-2}\right)^{2j+1} \frac{m-1}{m-2}(d(x,y))^{2j+1} (m-1)^{-d(x,y)}.
\end{align*}

\end{proof}

\subsection{Bound on the higher derivatives of $\Gamma(x,y|\z)$}\label{subsec:derivative}
In this section we obtain bound for the higher derivatives of the function $g(\z)=\Gamma(x,y|\z)$ evaluated at the point $\z=1$. Recall from~\eqref{eq:lem_1.24} that for $x,\,y\in\T_m$
\begin{align*}
g(\z)=\Gamma(x,y|\z) = \frac{2(m-1)}{m-2+\sqrt{m^2-4(m-1)\z^2}}\left( \frac{m-\sqrt{m^2-4(m-1)\z^2}}{2(m-1)\z}\right)^{d(x,y)}, \quad\z\in\mathbb{C}.
\end{align*}
We prove the following bound.
\begin{lemma}
Let $x,y\in\T_m$ and $d=d(x,y)$. Then for $k\ge 1$
\begin{align}\label{eq:g_deriv_bound}
g^{(k)}(1) \le 3^{(k-1)^2} (k-1)^{k-1}\left(\frac{4(m-1)}{m-2}\right)^k \frac{m-1}{m-2}d^k (m-1)^{-d},
\end{align}
and
\begin{align*}
g(1)= \frac{m-1}{m-2} (m-1)^{-d}.
\end{align*}
\end{lemma}
\begin{proof}
We write $\rho(\z):=\sqrt{m^2-4(m-1)\z^2}$. Then
\begin{align}\label{eq:g(z)}
g(\z)= \frac{2(m-1)}{m-2+\rho(\z)}\left( \frac{m-\rho(\z)}{2(m-1)\z}\right)^{d}.
\end{align}
We have
\begin{align*}
g(1)= \frac{m-1}{m-2} (m-1)^{-d}.
\end{align*}
Taking logarithms on both sides of~\eqref{eq:g(z)} and then differentiating we get
\begin{align*}
\frac{g{'}(\z)}{g(\z)} &= \frac{4(m-1)\z}{(m-2+\rho(\z))\rho(\z)} + \frac{d4(m-1)\z}{(m-\rho(\z))\rho(\z)} - \frac{d}{\z}\\
&=: h(\z).
\end{align*}
Note that here we have used $\rho{'}(\z)=-(4(m-1)\z)/\rho(\z)$. So we have 
\begin{align*}
g{'}(\z)=g(\z)h(\z).
\end{align*}
To obtain bounds for the derivatives of $g$ we first bound $h$ and its derivatives evaluated at $\z=1$. We have
\begin{align*}
h(1)&= \frac{4(m-1)}{2(m-2)^2} + \frac{d4(m-1)}{2(m-2)} - d \le \frac{4(m-1)}{m-2} d.
\end{align*}
Differentiating $h$ we get
\begin{align*}
h{'}(\z) &= \frac{4(m-1)}{(m-2+\rho(\z))\rho(\z)} \left[ 1- \frac{\z\rho{'}(\z)}{m-2+\rho(\z)} - \frac{\z\rho{'}(\z)}{\rho(\z)}\right]\\
\qquad\qquad &+ \frac{d4(m-1)}{(m-\rho(\z))\rho(\z)} \left[ 1+ \frac{\z\rho{'}(\z)}{m-\rho(\z)} - \frac{\z\rho{'}(\z)}{\rho(\z)}\right] -\frac{d}{\z^2}.
\end{align*}
Note that $\rho(1)=m-2$ and $\rho{'}(1)= - (4(m-1))/(m-2)$. Using these values we have
\begin{align}
h{'}(1)&= \frac{4(m-1)}{2(m-2)^2} \left[1+ \frac{4(m-1)}{2(m-2)^2} + \frac{4(m-1)}{(m-2)^2} \right]\nonumber\\
&\qquad\qquad + \frac{d4(m-1)}{2(m-2)} \left[1-\frac{4(m-1)}{2(m-2)} +\frac{4(m-1)}{(m-2)^2}\right] -d\nonumber\\
& \le 3 \frac{4(m-1)}{m-2} \left[ \frac{4(m-1)}{2(m-2)^2} + \frac{d4(m-1)}{2(m-2)} \right]\nonumber\\
& \le 3 \left(\frac{4(m-1)}{m-2}\right)^2 d.\label{eq:h'_bounded}
\end{align}
To obtain a bound on $h^{''}(1)$ we write $h{'}(\z)$ as
\begin{align*}
h{'}(\z) & =  \frac{4(m-1)}{(m-2+\rho(\z))\rho(\z)} + \frac{(4(m-1))^2\z^2}{(m-2+\rho(\z))^2\rho(\z)^2} + \frac{(4(m-1))^2\z^2}{(m-2+\rho(\z))\rho(\z)^3}\\
& + \frac{d4(m-1)}{(m-\rho(\z))\rho(\z)}  - \frac{d(4(m-1))^2\z^2}{(m-\rho(\z))^2\rho(\z)^2} + \frac{d(4(m-1))^2\z^2}{(m-\rho(\z))\rho(\z)^3} - \frac{d}{\z^2}.
\end{align*}
Now differentiating with respect to $\z$ we obtain
\begin{align*}
h^{''}(\z)&= \frac{4(m-1)}{(m-2+\rho(\z))\rho(\z)} \left[ \frac{4(m-1)\z}{(m-2+\rho(\z))\rho(\z)} + \frac{4(m-1)\z}{\rho(\z)^2}\right]\\
&  + \frac{(4(m-1))^2}{(m-2+\rho(\z))^2\rho(\z)^2} \left[ 2\z + \frac{2(4(m-1))\z^3}{(m-2+\rho(\z))\rho(\z)} + \frac{2(4(m-1))\z^3}{\rho(\z)^2} \right] \\
& +  \frac{(4(m-1))^2}{(m-2+\rho(\z))\rho(\z)^3}\left[ 2\z + \frac{(4(m-1))\z^3}{(m-2+\rho(\z))\rho(\z)} + \frac{3(4(m-1))\z^3}{\rho(\z)^2} \right]\\
& +\frac{d4(m-1)}{(m-\rho(\z))\rho(\z)} \left[ -\frac{4(m-1)\z}{(m-\rho(\z))\rho(\z)} + \frac{4(m-1)\z}{\rho(\z)^2} \right]\\
&-\frac{d(4(m-1))^2}{(m-\rho(\z))^2\rho(\z)^2} \left[ 2\z - \frac{2(4(m-1))\z^3}{(m-\rho(\z))\rho(\z)} + \frac{2(4(m-1))\z^3}{\rho(\z)^2} \right]\\
& + \frac{d(4(m-1))^2}{(m-\rho(\z))\rho(\z)^3} \left[ 2\z - \frac{4(m-1)\z^3}{(m-\rho(\z))\rho(\z)} + \frac{3(4(m-1))\z^3}{\rho(\z)^2} \right] + \frac{2d}{\z^3}.
\end{align*}
Putting $\z=1$ we have
\begin{align*}
h^{''}(1)&= \frac{4(m-1)}{2(m-2)^2} \left[ \frac{4(m-1)}{2(m-2)^2} + \frac{4(m-1)}{(m-2)^2}\right]\\
&  + \frac{(4(m-1))^2}{4(m-2)^4} \left[ 2 + \frac{2(4(m-1))}{2(m-2)^2} + \frac{2(4(m-1))}{(m-2)^2} \right] \\
& +  \frac{(4(m-1))^2}{2(m-2)^4}\left[ 2 + \frac{(4(m-1))}{2(m-2)^2} + \frac{3(4(m-1))}{(m-2)^2} \right]\\
& +\frac{d4(m-1)}{2(m-2)} \left[ -\frac{4(m-1)}{2(m-2)} + \frac{4(m-1)}{(m-2)^2} \right]\\
&-\frac{d(4(m-1))^2}{4(m-2)^2} \left[ 2 - \frac{2(4(m-1))}{2(m-2)} + \frac{2(4(m-1))}{(m-2)^2} \right]\\
& + \frac{d(4(m-1))^2}{2(m-2)^3} \left[ 2 - \frac{4(m-1)}{2(m-2)} + \frac{3(4(m-1))}{(m-2)^2} \right] + 2d.
\end{align*}
We observe that the term inside the square bracket in each summand is bounded by $(9(4(m-1)))/(m-2)$ and the other terms are the same as the summands in $h{'}(1)$ except for the last term. So we conclude using~\eqref{eq:h'_bounded} that
\begin{align*}
h^{''}(1) \le 9\cdot 3 \left(\frac{4(m-1)}{m-2}\right)^3 d \le 9\cdot 4 \left(\frac{4(m-1)}{m-2}\right)^3 d.
\end{align*}
In a similar way we can write $h^{(k)}(1)$ so that the term inside the square bracket in each summand is bounded by $(3 (2k-1)(4(m-1)))/(m-2)$ and the other terms are the same as the summands in $h^{(k-1)}(1)$ except the last term. Hence we conclude that 
\begin{align*}
h^{(k)}(1) &\le 3^k(1\cdot 3 \cdot 5\cdots (2k-1)) \left(\frac{4(m-1)}{m-2}\right)^{(k+1)} d\\
&\le 3^k k^k\left(\frac{4(m-1)}{m-2}\right)^{(k+1)} d,
\end{align*}
where we obtain the second inequality by using the relation between the arithmetic and the geometric mean.
We now prove~\eqref{eq:g_deriv_bound} by the method of induction. We have
\begin{align*}
g^{(1)}(1) = g(1) h(1) \le \frac{m-1}{m-2} (m-1)^{-d}\frac{4(m-1)}{m-2} d.
\end{align*}
Assume that~\eqref{eq:g_deriv_bound} holds true for $k=1,\,\ldots,\,\ell-1$.
Now we have
\begin{align*}
g^{(\ell)}(1) &= (gh)^{(\ell-1)}(1)=\sum_{k=0}^{\ell-1}\begin{pmatrix} \ell-1\\ k\end{pmatrix} g^{(\ell-1-k)}(1) h^{(k)}(1)\\
&\le \sum_{k=0}^{\ell-2} \begin{pmatrix} \ell-1\\ k\end{pmatrix}\left[ 3^{(\ell-1-k-1)^2}(\ell-1-k-1)^{\ell-1-k-1}\left(\frac{4(m-1)}{m-2}\right)^{\ell-1-k}\frac{m-1}{m-2}d^{\ell-1-k}\right.\\  
&\left.(m-1)^{-d}\right]3^kk^k\left(\frac{4(m-1)}{m-2}\right)^{(k+1)} d + \frac{m-1}{m-2}d (m-1)^{-d}3^{\ell-1}(\ell-1)^{\ell-1}\left(\frac{4(m-1)}{m-2}\right)^{\ell} d \\
&\le \left[ 3^{(\ell-2)^2} (\ell-1)^{\ell-1}\left(\frac{4(m-1)}{m-2}\right)^{\ell} \frac{m-1}{m-2}d^{\ell} (m-1)^{-d}\right] \sum_{k=0}^{\ell-1} \begin{pmatrix} \ell-1\\ k\end{pmatrix}\\
&= \left[ 3^{(\ell-2)^2} (\ell-1)^{\ell-1}\left(\frac{4(m-1)}{m-2}\right)^{\ell} \frac{m-1}{m-2}d^{\ell} (m-1)^{-d}\right] 2^{\ell-1}\\
& \le 3^{(\ell-1)^2} (\ell-1)^{\ell-1}\left(\frac{4(m-1)}{m-2}\right)^{\ell} \frac{m-1}{m-2}d^{\ell} (m-1)^{-d}.
\end{align*}
Therefore by induction~\eqref{eq:g_deriv_bound} holds for all $k\ge 1$.
\end{proof}

\paragraph*{Availability of data and materials} Data sharing not applicable to this article as no datasets were generated or analysed during the current study.



\appendix
\section{An alternative argument for~\eqref{eq:E_n_bound}}\label{app:alternative}
After the revision of the paper it was pointed out to us by an anonymous referee that an alternative proof can be carried out to quantitatively estimate the error one commits by replacing $G_n$ by $G$. This proof gives a bound comparable to~\eqref{eq:E_n_bound} for points that are far away from the boundary. For completeness we would like to outline this proof here.
\begin{proof}
The proof is based on a double application of the maximum principle for harmonic functions~\cite[Theorem 1.37]{Barlow2017Feb}. Fix $y\in V_n$. We define the function $H_y(\cdot)$ as
\begin{eqnarray*}
V_n&\to& \R\\
x&\mapsto & G_{n}(x,y)-G(x,y).
\end{eqnarray*}
We then set
\[
a:=\sup_{x\in V_{n-1}^{\c}}\left|H_{y}(x)\right|.
\]
We have that 
\[
\begin{cases}
\left|\Delta H_{y}(x)\right|\le2a & x\in V_{n}^{\c}\\
\Delta(\Delta H_{y})(x)=0 & x\in V_{n}
\end{cases}
\]
so that $\text{\ensuremath{\Delta H_{y}}(\ensuremath{\cdot})}$ is
harmonic in $V_{n}$. We can invoke the the maximum principle to say that $\displaystyle\max_{x\in V_{n}}\left|\Delta H_{y}(x)\right|\le 2a.$
Now consider the function
\[
f(x):=a+\frac{2am}{m-2}d(x,\,\partial_{1}V_{n}).
\]
It is clear that $f(x)=a$ on $V_{n}^{\c}$, and moreover that for
$x\neq o$
\begin{eqnarray*}
\Delta f(x) & = & \frac{2a}{m-2}\left[(m-1)\left(d(x,\,\partial_{1}V_{n})-1\right)+\left(d(x,\,\partial_{1}V_{n})+1\right)-md(x,\,\partial_{1}V_{n})\right]\\
 & = & \frac{2a}{m-2}(2-m)=-2a
\end{eqnarray*}

while for $x=o$ we have $\Delta f(x)=-2am/(m-2)\le-2a.$ So the function
\[
H_{y}(\cdot)-f(\cdot)
\]

is subharmonic in $V_{n}$ and again by the maximum principle
\[
\max_{x\in V_{n}}\left|H(\cdot)-f(\cdot)\right|=\max_{x\in V_{n+1}\setminus V_{n}}H_{y}(\cdot)-f(\cdot)\le0
\]
since $\left|H_{y}(x)\right|\le a=f(x)$ by the definition of $a$
for $x\in V_{n+1}\setminus V_{n}.$ Running the same argument for $-f$
rather than $f$ we finally obtain that $\left|H_{n}(x)\right|\le f(x)$
in $V_{n}.$ 

This implies that for $x\in V_{n}$
\begin{eqnarray*}
\left|H_{n}(x)\right| & = & \left|G_{n}(x,y)-G(x,y)\right|\le a\left(1+C(m)d(x,\,\partial_{1}V_{n})\right)\\
 & \le & C'(m)d(x,\,\partial_{1}V_{n})\sup_{x'\in V_{n-1}^{\c}}\left|G(x',y)\right|\\
 &\le &\frac{C'(m)d(x,\,\partial_{1}V_{n})d(y,\,\partial_{1}V_{n})}{(m-1)^{d(y,\,\partial_{1}V_{n})}}.
\end{eqnarray*}
Being our argument symmetric in $x$ and $y$, we can conclude our result.
\end{proof}

\end{document}